\documentclass{amsart}
\pdfoutput=1
\usepackage[backend=biber,maxbibnames=5,maxalphanames=5,style=alphabetic,bibencoding=utf8,giveninits,url=false,isbn=false]{biblatex}
\AtBeginBibliography{\small}

\usepackage{lineno}
\usepackage[all,hyperref]{paper_diening}

\usepackage{pgfplots}
\usepackage{tikz}
\usepackage{pgf}

\newcommand{\nodes}{\mathcal{N}}
\newcommand{\vertices}{\mathcal{V}}
\newcommand{\verticesint}{\mathcal{V}^\circ}
\newcommand{\nodest}{\nodes_\bft}
\newcommand{\nodesint}{\nodes^\circ}

\newcommand{\nablax}{\nabla_\bfx}
\newcommand{\partialt}{\partial_\bft}
\newcommand{\triat}{\mathcal{T}_\bft}
\newcommand{\triax}{\mathcal{T}_\bfx}
\newcommand{\It}{\Pi_\bft}
\newcommand{\Ix}{\Pi_\bfx}
\newcommand{\Kx}{{K_\bfx}}
\newcommand{\Kt}{{K_\bft}}

\newcommand{\tria}{\mathcal{T}}

\newcommand{\tpsi}{\widetilde{\psi}}
\newcommand{\tPi}{\Pi}

\newcommand{\Pizero}{\Pi_0}

\newcommand{\dx}{\,\mathrm{d}x}
\newcommand{\ds}{\,\mathrm{d}s}
\newcommand{\dsig}{\,\mathrm{d}\sigma}

\providecommand{\RTk}{\mathcal{RT}_k(\tria)}

\providecommand{\That}{{\widehat{T}}}

\providecommand{\skpTtmp}[3]{{\ensuremath{#1\langle {#2}, {#3} #1\rangle_{\That}}}}
\providecommand{\skpT}[2]{\skpTtmp{}{#1}{#2}}

\providecommand{\interC}{\mathcal{I}}
\renewcommand{\bft}{t}
\renewcommand{\bfx}{x}

\makeatletter
\@namedef{subjclassname@2020}{%
  \textup{2020} Mathematics Subject Classification}
\makeatother

\usepackage{etoolbox} 

\begin{document}

\author[L.\ Diening]{Lars Diening}
\author[J.\ Storn]{Johannes Storn}
\author[T.\ Tscherpel]{Tabea Tscherpel}

\address[L.\ Diening, J.\ Storn]{Department of Mathematics, Bielefeld University, Postfach 10 01 31, 33501 Bielefeld, Germany}
\email{lars.diening@uni-bielefeld.de}
\email{jstorn@math.uni-bielefeld.de}
\address[T.\ Tscherpel]{Department of Mathematics, TU Darmstadt, Dolivostraße 15, 64293 Darmstadt, Germany}
\email{tscherpel@mathematik.tu-darmstadt.de}
\thanks{
The work of the authors was supported by the Deutsche Forschungsgemeinschaft (DFG, German
Research Foundation) – SFB 1283/2 2021 – 317210226.}

\subjclass[2020]{
 	65D05, 
    65M15, 
	65N15, 
    65N30, 
}
\keywords{Scott-Zhang, interpolation, projection, dual norm, negative Sobolev space, singular data}

\title{Interpolation Operator on negative Sobolev Spaces}

\begin{abstract}
  We introduce a Scott--Zhang type projection operator mapping to Lagrange elements for arbitrary polynomial order.  In addition to the usual properties, this operator is compatible with duals of first order Sobolev spaces.  More specifically, it is stable in the corresponding negative norms and allows for optimal rates of convergence.  We discuss alternative operators with similar properties.  As applications of the operator we prove interpolation error estimates for parabolic problems and smoothen rough right-hand sides in a least squares finite element method.
\end{abstract}
\maketitle
\section{Introduction}
\label{sec:introduction}
Interpolation operators are an important tool in numerical analysis and thus have been studied intensively \cite{Clement75,ScottZhang90,Carstensen99,Melenk05,DieningRuszicka07,ErnGuermond17}. 
While the focus of these works lies primarily on error control in $W^{s,p}(\Omega)$ for non-negative orders $s\in \mathbb{N}_0$, less is known for error control in negative Sobolev spaces $W^{-1,p}(\Omega)$ with $p\in (1,\infty)$. Such results are, however, needed for recent developments including
\begin{itemize}
\item time-stepping and space-time finite element methods for parabolic problems, where the time-derivative can only be controlled in this norm (cf.~\cite{Steinbach15,TantardiniVeeser16,StevensonWesterdiep20}),
\item multilevel decompositions  (cf.~\cite{BPV.2000,WZ.2017,F.2021}) and operator preconditioning (cf.~\cite{SV.2020,Sv.2021}),
\item regularizations for PDEs with rough right-hand side, in particular for methods that require the right-hand side to be in $L^p(\Omega)$ like some DG, DPG, non-conforming, or least squares schemes  (cf.~\cite{VeeserZanotti18,VeeserZanotti18b,VeeserZanotti19,FHK.2022}),
\item a~posteriori error control for rough right-hand sides (cf.~\cite{KreuzerVeeser21}).
\end{itemize}
This paper addresses this issue by designing a local Scott--Zhang type projection operator $\Pizero$ (resp.\ $\Pi$) onto the Lagrange finite element space of arbitrary polynomial degree $k\in \mathbb{N}$. 
The operator is stable and achieves optimal rates of convergence with respect to Sobolev norms of order larger or equal to minus one as stated in our main result in Theorem~\ref{thm:interpolOperator}. 

The construction of the operator and its analysis is carried out in Section~\ref{sec:higher-order} and generalizes Tantardini's approach for the lowest-order case $k=1$ \cite[Sec.~7.6]{Tantardini12}. 
We modify the classical design of the celebrated Scott--Zhang projection operator \cite{ScottZhang90} by using inner products with weights suited for functions in negative norm Sobolev spaces.  In particular, our weights are continuous, which allows testing with dual Sobolev functions. 
Furthermore, the design leads to an adjoint operator $\Pizero^*$ that preserves constants on interior simplices. 
Both properties are not available for most classical interpolation operators (see Section~\ref{sec:AlternativeOp}) but essential for the proof of stability and optimal convergence in Section~\ref{subsec:ProofThm1}. 
In Section~\ref{sec:local-basis} we define the weight functions locally by using polynomials of higher polynomial degree and in Section~\ref{sec:global-basis} we use them to obtain a global biorthogonal system. 
Section~\ref{sec:BoundaryCorr} displays a modification of $\Pizero$ in the vicinity of the boundary. 
This leads to an operator $\Pi$ with optimal approximation properties for functions without zero boundary traces.

In Section~\ref{sec:AlternativeOp} we compare the operator $\Pizero$ to existing operators suited for the approximation in the negative norm Sobolev space. 
For some applications it is beneficial to have a self-adjoint operator, which our operator is not.   
However, this property is not compatible with the operator being a local projection as demonstrated in Section~\ref{sec:SelfAdjProj}. 
Therefore, in Section~\ref{sec:SelfAdjLoc} we introduce an alternative local and self-adjoint operator, that is, however, not a projection. This reduces the maximal rate of convergence by one. 
Section~\ref{sec:LocalProj} illustrates that local projection operators with the beneficial properties of $\Pizero$ need local weight functions  of higher polynomial degree. 
This justifies the more involved design of $\Pizero$ in Section~\ref{sec:higher-order}. 

In Section~\ref{sec:applications} we present several applications of the novel operators. 
In Sections~\ref{sec:InterpolSemiDiscr} and~\ref{sec:InterPolTensor} we investigate interpolation error estimates for parabolic problems. 
Section~\ref{sec:InterpolSemiDiscr} utilizes recent results in \cite{TantardiniVeeser16} showing quasi-optimality of semi-discrete numerical schemes. 
Applying the operator $\Pizero$ leads to an a~priori error estimate for arbitrary polynomial degrees $k\in \mathbb{N}$, 
similar to the one obtained for $k = 1$ in \cite[Prop.~7.27]{Tantardini12}.  

For parabolic problems simultaneous space-time finite element methods represent an alternative to standard time-marching schemes. 
Such space-time schemes treat the time as an additional spatial dimension and apply techniques available for time-independent problems, see for example  \cite{Steinbach15,LangerMooreNeumueller16,StevensonWesterdiep20,DieningStorn21}.
In Section~\ref{sec:InterPolTensor} we construct an interpolation operator based on $\Pizero$ that improves existing interpolation error estimates on tensor product meshes significantly. 
In Section~\ref{sec:SmoothingRHS} we apply the operator $\Pi$ to the right-hand side of a least squares finite element method. 
This leads to a computable scheme even for right-hand sides $f\not\in L^2(\Omega)$. Based on the properties of $\Pi$ we conclude optimal rates of convergence.
The appendix displays results on approximation theory in Bochner spaces used in Section~\ref{sec:InterPolTensor}.

\subsubsection*{Notation}
 Throughout this paper we use the following notation. 
We suppose that $\Omega \subset \setR^d$ is a bounded, polyhedral domain with a shape regular triangulation~$\tria$ consisting of closed simplices. 
For any simplex $T \in \tria$ we denote the neighboring patch of $T$ by $\omega_T \coloneqq \bigcup \lbrace T'\in \tria\colon T \cap T' \neq \emptyset\rbrace$ and the local mesh size $h_T \coloneqq \textup{diam}(T)$. 
We denote by $\mathcal{V}$ the set of vertices in $\tria$. 
For any vertex $j \in \mathcal{V}$ we define the closed vertex patches by $\omega_j \coloneqq \bigcup\lbrace T\in \tria\colon j \in T\rbrace$ and $\omega_j^2 \coloneqq \bigcup\lbrace T \in \tria\colon T \cap \omega_j \neq \emptyset\rbrace$, and the local mesh size by $h_j \coloneqq \textup{diam}(\omega_j)$.

By~$L^p(\Omega)$ and $W^{s,p}(\Omega)$ we denote the standard Lebesgue and Sobolev space of order~$s\in \mathbb{N}_0$ and integrability~$p \in [1,\infty]$. 
Whenever $m,s$ are indices for the differentiability, then they are integers.
The set $W^{1,p}_0(\Omega)$ denotes the subset of~$W^{1,p}(\Omega)$-functions with zero boundary traces. 
In the following let $p'$ be the H\"older exponent of $p \in [1,\infty]$, defined by  $1/p +1/p' =1$. 
For $p \in (1,\infty]$ we denote by $W^{-1,p}(\Omega) \coloneqq (W^{1,p'}_0(\Omega))^*$ the dual Sobolev space. 
We write $\skp{f}{g}_\Omega$ both for the $L^2$-scalar product $\int_\Omega f(x)\cdot g(x)\dx$ and for the dual pairing $\skp{f}{g}_{W^{-1,p}(\Omega),W_0^{1,p'}(\Omega)}$.

Let $\mathcal{P}_k(T)$ denote the space of polynomials of degree at most~$k$ on $T\in \tria$. 
Let~$\mathcal{L}^1_k(\tria)$ for $k \in \setN$ denote the space of Lagrange functions of order~$k$, i.e.,\ $W^{1,1}$-functions whose restriction to each $d$-simplex $T\in \tria$ is in~$\mathcal{P}_k(T)$. 
Let $\mathcal{L}^1_{k,0}(\tria) \coloneqq \mathcal{L}^1_k(\tria) \cap W^{1,1}_0(\Omega)$ denote the subset of discrete functions with zero boundary traces.

If there exist generic constants $0< c \leq C<\infty$ that might depend on fixed parameters as the polynomial degree $k\in \mathbb{N}$, the domain $\Omega$, and the shape regularity of $\tria$, but are independent of any further quantities, we write $A \lesssim B$ for $A \leq CB$ and $A \eqsim B$ for $c B \leq A \leq C B$.

\section{Projection operator}
\label{sec:higher-order}
In this section we design an interpolation operator that is well-defined for functions in negative Sobolev spaces. 
The operator is a projection and has useful additional properties including (local) stability and approximation properties. 
\subsection{Main results: Properties of the projection operator}
In Sections~\ref{sec:local-basis}--\ref{sec:BoundaryCorr} we present for arbitrary $k\in \mathbb{N}$ Scott--Zhang type interpolation operators 
\begin{align*}
  \Pizero\colon W^{-1,2}(\Omega) \to \mathcal{L}_{k,0}^1(\tria)
   \qquad\text{and}\qquad
    \Pi\colon W^{-1,2}(\Omega) \to \mathcal{L}_{k}^1(\tria).
\end{align*}
We shall see that those operators are defined on the larger dual space $(W^{1,\infty}_0(\Omega))^*$, and hence in particular on $W^{-1,p}(\Omega)$ for any $p \in (1, \infty]$. 
Compared to existing operators their main advantages are stability and approximation properties including ones with respect to negative Sobolev norms. 
They represent a generalization of the lowest order case $k = 1$ introduced in \cite[Sec.~7.6]{Tantardini12}. 
All results are presented for scalar functions but obviously extend to vector-valued functions. 
 
\begin{theorem}[Main result]
  \label{thm:interpolOperator}
  Let $k \in \mathbb{N}$ be arbitrary.   
  The operators~$\Pizero$ and $\Pi$ are linear projections onto~$\mathcal{L}^1_{k,0}(\tria)$ and $\mathcal{L}^1_k(\tria)$, respectively. In addition, they have the following properties, for any $j\in \mathcal{V}$ and $T\in \tria$: 
  
  \noindent%
  \textbf{\,Localization of norms.} For $p \in (1,\infty)$ 
      there holds
  \begin{align}\label{eq:ApxEqui}
       \norm{\xi - \Pizero \xi }_{W^{-1,p}(\Omega)} & \eqsim \Big(\sum_{j \in \mathcal{V}} \norm{\xi- \Pizero \xi}_{W^{-1,p}(\omega_j)}^p \Big)^{1/p}\quad\text{for all }\xi \in W^{-1,p}(\Omega).
  \end{align}
  
  \noindent%
  \textbf{\,Approximabilty.} For $p \in (1,\infty)$, $q \in [1,\infty]$ and $0 \leq m \leq s \leq k+1$ one has
  \begin{align}\label{eq:ApxIloc}
    \begin{alignedat}{2}
      \norm{v - \Pizero v }_{W^{-1,p}(\omega_j)} & \lesssim   h_j \norm{v}_{L^{p}(\omega_j^2)}  &&\text{for all } v \in L^{p}(\Omega),\\
      \norm{w - \Pizero w }_{W^{-1,p}(\omega_j)}&\lesssim  h_j^{s+1} \norm{\nabla^s w}_{L^p(\omega^2_j)} 
      \quad 
      &&\text{for all } w\in W^{1,p}_0(\Omega)\cap W^{s,p}(\Omega),
      \\
      \norm{\nabla^m (w - \Pizero w)}_{L^q(T)} &\lesssim h_T^{s-m} \norm{ \nabla^s w }_{L^q(\omega_T)}
     \quad &&\text{for all } w \in W^{1,q}_0(\Omega)\cap W^{s,q}(\Omega).
    \end{alignedat}
  \end{align}
  
  \noindent%
  \textbf{\,Local Stability.}  For $p \in (1,\infty)$, $q \in [1,\infty]$ and $0 \leq s \leq k$ there holds
  \begin{align}
    \begin{alignedat}{2}
      \norm{\Pizero \xi}_{W^{-1,p}(\omega_j)} &\lesssim \norm{\xi}_{W^{-1,p}(\omega_j^2)} &&\text{for all }\xi \in W^{-1,p}(\Omega),
      \\
      \lVert \Pizero v \rVert_{L^q(T)}& \lesssim \lVert v \rVert_{L^q(\omega_T)} &&\text{for all }v\in L^q(\Omega),
      \\
      \lVert \nabla^s\Pizero w\rVert_{L^q(T)}& \lesssim \lVert \nabla^s w \rVert_{L^q(\omega_T)} \qquad &&\text{for all }w\in W^{1,q}_0(\Omega)\cap W^{s,q}(\Omega).
    \end{alignedat}
  \end{align}

  \noindent%
  \textbf{\,Global Stability.}   For $p \in (1,\infty)$, $q \in [1,\infty]$ and $0 \leq s \leq k$ one has
  \begin{align}
    \begin{alignedat}{2}
      \norm{\Pizero \xi}_{W^{-1,p}(\Omega)} &\lesssim \norm{\xi}_{W^{-1,p}(\Omega)} &\qquad&\text{for all }\xi \in W^{-1,p}(\Omega),
      \\
      \lVert \Pizero v \rVert_{L^q(\Omega)}& \lesssim \lVert v \rVert_{L^q(\Omega)} &&\text{for all }v\in L^q(\Omega),
      \\
      \lVert \nabla^s\Pizero w \rVert_{L^q(\Omega)}& \lesssim \lVert \nabla^s w \rVert_{L^q(\Omega)}&&\text{for all }w\in W^{1,q}_0(\Omega)\cap W^{s,q}(\Omega).
    \end{alignedat}
  \end{align}
  All statements of the theorem are valid for the projection $\Pi$, even for $w \in W^{1,q}(\Omega)$ and $w \in W^{1,p}(\Omega)$ instead of $w \in W^{1,q}_0(\Omega)$ and $w \in W^{1,p}_0(\Omega)$, respectively. 
  The hidden constants are independent of $p$ and~$q$.
\end{theorem}
A key in the proof of Theorem~\ref{thm:interpolOperator} are certain properties of the $L^2$-adjoint operators $\Pizero^*$ and $\Pi^*$ of $\Pizero$ and $\Pi$, respectively, presented in the following theorem.
\begin{theorem}[Adjoint operators]\label{thm:PropPi*}
Let $\Pizero$ and $\Pi$ be the projection operators in Theorem~\ref{thm:interpolOperator}. 
  The $L^2$-adjoint projection $\Pizero^*$ maps 
  $(W_0^{1,\infty}(\Omega))^*$ onto a subset of $\mathcal{L}^1_{3k,0}(\tria)$ and satisfies the following properties, for any $j \in \vertices$ and $T \in \tria$: 

  \noindent%
  \textbf{\,Preservation of constants.} $\Pizero^*$ preserves constants on interior simplices, i.e.,
  \begin{align}\label{eq:Adj1}
    (\Pizero^*1)|_T = 1 \qquad \text{ for all }\, T\in \tria \text{ with }T\cap \partial \Omega = \emptyset.
  \end{align}

  \noindent%
  \textbf{\,Localization of norms.} For any $p \in (1,\infty)$ 
      there holds
  \begin{align}\label{eq:ApxEquidual}
       \norm{\xi - \Pizero^* \xi }_{W^{-1,p}(\Omega)} & \eqsim \Big(\sum_{j \in \mathcal{V}} \norm{\xi- \Pizero^* \xi}_{W^{-1,p}(\omega_j)}^p \Big)^{1/p}\quad\text{for all }\xi \in W^{-1,p}(\Omega).
  \end{align}

  \noindent%
  \textbf{\,Approximabilty.} For $p \in (1,\infty)$, $q \in [1,\infty]$ and $0 \leq m \leq s \leq 1$ one has
  \begin{align}
    \begin{alignedat}{2}
      \norm{v - \Pizero^* v }_{W^{-1,p}(\omega_j)} & \lesssim   h_j \norm{v}_{L^{p}(\omega_j^2)}  &\quad&\text{for } v \in L^{p}(\Omega),\\
      \norm{w - \Pizero^* w }_{W^{-1,p}(\omega_j)}&\lesssim  h_j^{s+1} \norm{\nabla^s w}_{L^p(\omega^2_j)} 
       &&\text{for } w\in W^{1,p}_0(\Omega)\cap W^{s,p}(\Omega),
      \\
      \norm{\nabla^m (w - \Pizero^* w)}_{L^q(T)} &\lesssim h_T^{s-m} \norm{ \nabla^s w }_{L^q(\omega_T)}&&\text{for } w \in W^{1,q}_0(\Omega)\cap W^{s,q}(\Omega).
    \end{alignedat}
  \end{align}

  \noindent%
  \textbf{\,Local Stability.}  For $p \in (1,\infty)$ and $q \in [1,\infty]$ there holds
  \begin{align}
    \begin{alignedat}{2}
      \norm{\Pizero^* \xi}_{W^{-1,p}(\omega_j)} &\lesssim \norm{\xi}_{W^{-1,p}(\omega_j^2)} &\qquad&\text{for all }\xi \in W^{-1,p}(\Omega),
      \\
      \lVert \Pizero^* v \rVert_{L^q(T)}& \lesssim \lVert v \rVert_{L^q(\omega_T)} &&\text{for all }v\in L^q(\Omega),
      \\
      \lVert \nabla\Pizero^* w \rVert_{L^q(T)}& \lesssim \lVert \nabla w \rVert_{L^q(\omega_T)}&&\text{for all }w\in W^{1,q}_0(\Omega).
    \end{alignedat}
  \end{align}

  \noindent%
  \textbf{\,Global Stability.}  For $p \in (1,\infty)$ and $q \in [1,\infty]$ one has
  \begin{align}
    \begin{alignedat}{2}
      \norm{\Pizero^* \xi}_{W^{-1,p}(\Omega)} &\lesssim \norm{\xi}_{W^{-1,p}(\Omega)} &\qquad&\text{for all }\xi \in W^{-1,p}(\Omega),
      \\
      \lVert \Pizero^* v \rVert_{L^q(\Omega)}& \lesssim \lVert v \rVert_{L^q(\Omega)} &&\text{for all }v\in L^q(\Omega),
      \\
      \lVert \nabla\Pizero^* w \rVert_{L^q(\Omega)}& \lesssim \lVert \nabla w \rVert_{L^q(\Omega)}&&\text{for all }w\in W^{1,q}_0(\Omega).
    \end{alignedat}
  \end{align}
  All statements of the theorem except localization and stability in $W^{-1,p}(\Omega)$ and $W^{-1,p}(\omega_j)$ are valid for the $L^2$-adjoint projection $\Pi^*\colon L^1(\Omega) \to \mathcal{L}^1_{3k,0}(\tria)$ as well. Still, the localization in~\eqref{eq:ApxEquidual} holds for functions in $L^p(\Omega)$.
   The hidden constants are independent of $p$ and $q$.
\end{theorem}

To prove Theorem~\ref{thm:interpolOperator}  and~\ref{thm:PropPi*} we locally construct in Section~\ref{sec:local-basis} a system of polynomials of degree $3k$ dual to the Bernstein polynomials of degree $k$. 
Those allow us to find a system of global biorthogonal functions in Section~\ref{sec:global-basis}. 
In Section~\ref{subsec:ProofThm1} we design the operator~$\Pizero$ mapping to $\mathcal{L}^1_{k,0}(\tria)$ and verify its properties. The key ingredient in the proof of approximation and stability results of $\Pi_0$ in negative Sobolev norms is the preservation of constants of the adjoint operator $\Pi_0^*$ in \eqref{eq:Adj1}, see Remark~\ref{rem:PresMass} below for further discussions. 
In Section~\ref{sec:BoundaryCorr} we introduce~$\Pi$, which modifies~$\Pizero$ close to the boundary and maps to $\mathcal{L}^1_k(\tria)$. 

\begin{remark}[Fractional orders]\label{rem:FracOrder} 
  Combining the localization of the norm and interpolating between the approximation estimates allows us to obtain estimates with fractional orders of convergence. 
\end{remark}
\begin{remark}[Other approaches]
   Classical Scott--Zhang type operators, such as the one in \cite{ScottZhang90}, use discontinuous weight functions and hence are not defined on $W^{-1,2}(\Omega)$.  
 In addition their adjoint operators do not preserve constants on interior simplices as in~\eqref{eq:Adj1}, which is a key in the proof of optimal convergence rates in Section~\ref{subsec:ProofThm1} below. 
Operators that are better suited for approximation in~$W^{-1,2}(\Omega)$ include the following. 
  \begin{itemize}
  \item In \cite[Thm.~5.2]{TantardiniVeeser16} Tantardini and Veeser design for arbitrary polynomial degree $k\in \mathbb{N}$ an operator with $W^{-1,2}(\Omega)$-approximability for~$W^{s,2}(\Omega)$-functions with $s=0,\dots, k+1$. 
  However, the operator is not well-defined on~$W^{-1,2}(\Omega)$.
  \item In \cite[Sec.~7.6]{Tantardini12} Tantardini introduces an operator for $k=1$ defined on~$W^{-1,2}(\Omega)$ with $W^{-1,2}$-approximability for functions in $W^{-1,2}(\Omega)$. 
This operator is well-defined on $W^{-1,2}(\Omega)$ and coincides with our lowest-order operator. 
Indeed, the reformulation of this operator inspires our generalization to~$k \geq 1$, see Remark~\ref{rmk:TV-reform}. 
  \item 
    In \cite{F.2021} F\"uhrer introduces several  operators for the analysis of multilevel preconditioning.  
    One of them is a projection onto the space~$\mathcal{L}^0_0(\tria)$ and allows  for~$W^{-1,2}(\Omega)$-approximability of functions in~$L^2(\Omega)$, see~\cite[Thm.~8]{F.2021}.  
    This operator is $W^{-1,2}(\Omega)$-stable, but not $W^{1,2}(\Omega)$-stable. 
\item In~\cite{SV.2020} Stevenson and van~Veneti\"e develop an alternative projection for the analysis of operator preconditioning. Their operator satisfies approximation and stability estimates in $W^{-1,2}(\Omega)$, but has a slightly increased support. For a more detailed comparison see Remark~\ref{rmk:op-Stevenson}. 
  \item In~\cite{KreuzerVeeser21}  Kreuzer and Veeser introduce an operator that is locally $W^{-1,2}$-stable with local $W^{-1,2}$-approximability.
The operator is a projection to piecewise constant functions enriched by the span of face-supported Dirac distributions. 
Thus, it does not map to a polynomial space but to true functionals (not representable by~$L^2$-densities).
\end{itemize}
A possible alternative technique to derive similar localized estimates as in Theorem~\ref{thm:interpolOperator} is discussed in \cite[Thm.~3]{Veeser17}. 
\end{remark}
\subsection{Local construction}
\label{sec:local-basis}

In this section we locally construct a system of polynomial functions of degree $3k$ that forms a biorthogonal system in conjunction with Bernstein polynomials of degree $k\in \mathbb{N}$. 

\subsubsection*{Bernstein basis}
Let $\nodes$ denote the set of Lagrange nodes of $\mathcal{L}^1_k(\tria)$, for $k \in \mathbb{N}$. 
As for the Lagrange basis we can assign each Bernstein basis function $b_i \in \mathcal{L}^1_k(\tria)$ to a Lagrange node~$i \in \nodes$. 
Its support $\omega_i \coloneqq \textup{supp}(b_i)\subset \overline{\Omega}$ consists of those $d$-simplices $T\in \tria$ that contain the Lagrange node~$i \in \nodes$, that is, $\omega_i = \bigcup \lbrace T\in \tria\colon i \in T\rbrace$. 
The Bernstein polynomials are best defined in terms of local barycentric coordinates~$\lambda = (\lambda_0, \ldots ,\lambda_d)$. 
To each Lagrange node~$i \in \nodes$ and each simplex $T \in \tria(\omega_i) \coloneqq \lbrace T' \in \tria\colon T\subset \omega_i\rbrace$ we assign a multi-index $ \alpha(i,T) =  \alpha \in \setN_0^{d+1}$ with $\abs{\alpha} \coloneqq \abs{\alpha_0} + \dots + \abs{\alpha_0} = k$ that naturally corresponds to the node~$i$.  
Then $b_i$ is an element in the space $\mathcal{P}_k(T)$ of polynomials with maximal degree $k$ and is locally given by
\begin{align}
  \label{eq:def-Bernstein}
  b_i|_T &= b_{\alpha(i,T)}  =  b_\alpha =
  \frac{\abs{\alpha}!}{\alpha!} \lambda^\alpha \in \mathcal{P}_k(T).
\end{align}
The set of Bernstein polynomials $\set{b_i\colon i \in \nodes}$ is a basis of~$\mathcal{L}^1_k(\tria)$ and $\set{b_i\colon i \in \nodesint}$ with $\nodesint \coloneqq \lbrace i \in \nodes\colon i \not\in \partial \Omega\rbrace$ is a basis of~$\mathcal{L}^1_{k,0}(\tria)$. 
Each function $b_i$ is non-negative, which is crucial for the following arguments. 
Note that locally on $T \in \tria$ we have
\begin{align}
  \label{eq:local-partition-of-unity}
  \sum_{\abs{\alpha}=k} b_\alpha
  &= 
    \sum_{\abs{\alpha}=k}   \frac{\abs{\alpha}!}{\alpha!} \lambda^\alpha  = \bigg( \sum_{j=0}^d \lambda_j \bigg)^k = 1.
\end{align}
This implies that $(b_i)_{i \in \nodes}$ is a partition of unity, that is,
\begin{align}
  \label{eq:global-partition-of-unity}
  \sum_{i \in \nodes} b_i &= 1,
\end{align}
which is essential for the following construction. 
For the local construction of a system of biorthogonal polynomials we consider the standard reference $d$-simplex $\That$ with volume $\abs{\That} = 1/d!$.   
In the following let $\alpha,\beta,\mu \in \setN_0^d$ be multi-indices with $\abs{\alpha}=\abs{\beta}=\abs{\mu}=k$.
The value $\skpT{b_\alpha}{1}$ is independent of~$\alpha$ and is denoted by
\begin{align}\label{eq:ck}
  c_k \coloneqq \skpT{b_\alpha}{1} = \frac{k!}{(d+k)!}.
\end{align}
\begin{lemma}[Local dual basis]
  \label{lem:local-basis}
  For each multi-index $\alpha \in \mathbb{N}_0^{d+1}$ with $\abs{\alpha} = k$ there exists a polynomial $p_\alpha \in \mathcal{P}_{3k}(\That)$ such that the following properties are satisfied.
  \begin{enumerate}
  \item \label{itm:local-product} For any $\alpha$ there exists a polynomial $z_{\alpha} \in \mathcal{P}_{2k}(\That)$ such that
  \begin{align*}
  p_{\alpha} = b_{\alpha} z_{\alpha}. 
  \end{align*}
  \item \label{itm:local-duality}
    For all multi-indices $\alpha,\beta \in \mathbb{N}^{d+1}_0$ with $\abs{\alpha}=\abs{\beta}=k$ we have
    \begin{align*}
      \skpT{p_{\alpha}}{b_{\beta}} = \delta_{\alpha,\beta}. 
    \end{align*}
  \item \label{itm:local-one} One has that
    \begin{align*}
      \sum_{\abs{\alpha} = k} p_{\alpha} = \frac{1}{c_k} = \frac{(d+k)!}{k!}. 
    \end{align*}
  \item \label{itm:local-symmetry}  The functions $z_{\alpha}$ and $p_\alpha$ are symmetric with respect to permutations of the barycentric coordinates.
  \end{enumerate}
\end{lemma}
\begin{proof}
Inspired by the operator introduced by Tantardini~\cite{Tantardini12} for~$k=1$ (see Remark~\ref{rmk:TV-reform} below), we use the ansatz, with functions~$q_\alpha \in \mathcal{P}_k(\That)$ to be determined,
  \begin{align}\label{eq:ansatz-p}
    p_{\alpha} &= b_{\alpha} z_{\alpha}  \quad \text{ with }\quad   
    z_{\alpha} =  \frac{(d+k)!}{k!} + q_{\alpha} - \sum_{\abs{\mu} = k} b_{\mu} q_{\mu}. 
  \end{align} 
  Then we have that $z_{\alpha} \in \mathcal{P}_{2k}(\That)$ and $p_{\alpha} \in \mathcal{P}_{3k}(\That)$ and thus~\ref{itm:local-product} is satisfied. 
  The property $\sum_{\abs{\alpha}=k}b_\alpha =1$ leads to
  \begin{align*}
    \sum_{\abs{\alpha}=k}
    p_{\alpha}
    &= \sum_{\abs{\alpha}=k} b_{\alpha} \bigg( \frac{(d+k)!}{k!} + q_{\alpha} - \sum_{\abs{\mu} = k} b_{\mu} q_{\mu} \bigg)
      = \frac{(d+k)!}{k!}.
  \end{align*}
  Hence, the property in~\ref{itm:local-one} results directly from our ansatz as well.
  The crucial step is to determine $q_\alpha$ such that~\ref{itm:local-duality} holds. 
  Let $N \coloneqq \dim \mathcal{P}_k(\That)$. For all vectors $(r_\mu)_{\abs{\mu}= k} \in \mathcal{P}_k(\That)^N$ we abbreviate the sum
  \begin{align*}
    \bar{r} & \coloneqq \sum_{\abs{\mu}=k} b_\mu r_\mu.
  \end{align*}
  With the definition in~\eqref{eq:ansatz-p} and with this notation,~\ref{itm:local-duality} is equivalent to
  \begin{align}
    \label{eq:goal1}
    \skpT{b_\alpha (q_\alpha- \bar{q})}{b_\beta} = \delta_{\alpha,\beta} - \frac{(d+k)!}{k!} \skpT{b_\alpha}{b_\beta} \qquad \text{for all }\abs{\alpha}=\abs{\beta}=k.
  \end{align}
  Let us define the linear operator~$S$ that maps vectors of polynomials $(r_\alpha)_{\abs{\alpha}=k}$ to the matrix $( \skpT{b_\alpha (r_\alpha- \bar{r})}{b_\beta})_{\alpha,\beta}$ in the sense that
  \begin{align*}
    S\colon \big(\mathcal{P}_k(\That)\big)^N \to \setR^{N \times N}\qquad\text{with}\qquad (r_\alpha)_{\abs{\alpha}=k} \mapsto (\skpT{b_\alpha (r_\alpha- \bar{r})}{b_\beta})_{\alpha,\beta}.
  \end{align*}  
  The vectors in $(\mathcal{P}_k(\That))^N$ have~$N^2$ degrees of freedom and the image has $N^2$ entries. 
  However, the operator~$S$ is neither injective nor surjective.

 The first isomorphism theorem states that the restriction $S\colon U \mapsto S(U)$ with quotient space $U = (\mathcal{P}_k(\That))^N / \ker(S)$ is an isomorphism. 	
Let $(r_{\alpha})_{\abs{\alpha} = k}$ be in the kernel of $S$, i.e., 
\begin{align}
	\skpT{b_\alpha(r_\alpha - \bar{r})}{b_\beta} &= 0 \qquad \text{for all }\alpha, \beta. 
\end{align}
Since $r_\alpha \in \linearspan \set{b_\beta\colon \abs{\beta}=k}$, it follows in particular that 
\begin{align}
	\label{eq:fill-diagonal3v}
	\skpT{b_\alpha(r_\alpha - \bar{r})}{r_\alpha} &= 0 \qquad \text{ for any } \alpha. 
\end{align}
By~\eqref{eq:fill-diagonal3v} and the partition of unity $\sum_{\abs{\alpha} = k} b_\alpha=1$ we have
\begin{align*}
	\sum_{\abs{\alpha} = k} \int_{\That} b_\alpha \abs{r_\alpha - \bar{r}}^2\dx
	&= \sum_{\abs{\alpha} = k} \skpT{b_\alpha r_\alpha}{r_\alpha} - 2\,\sum_{\abs{\alpha} = k} \skpT{b_\alpha \bar{r}}{r_\alpha} + \sum_{\abs{\alpha} = k} \skpT{b_\alpha \bar{r}}{\bar{r}}
	\\
	&= 
	 \skpT{\bar{r}}{\bar{r}}- \sum_{\abs{\alpha} = k} \skpT{b_\alpha r_\alpha}{r_\alpha}.
\end{align*}
Jensen's inequality and $\sum_{\abs{\alpha}=k} b_\alpha =1$ lead to the upper bound
\begin{align*}
	\skpT{\bar{r}}{\bar{r}}
	&= \int_{\That} \biggabs{\sum_{\abs{\alpha} = k} b_\alpha r_\alpha}^2\dx
	\leq \int_{\That} \sum_{\abs{\alpha} = k} b_\alpha \abs{r_\alpha}^2\dx = \sum_{\abs{\alpha} = k} \skpT{b_\alpha r_\alpha}{r_\alpha}.
\end{align*}
Applying this estimate in the previous equation shows that
\begin{align}
	\sum_{\abs{\alpha} = k} \int_{\That} b_\alpha \abs{r_\alpha - \bar{r}}^2\dx
	&\leq 0.\label{eq:fill-diagonal}
\end{align}
Since the Bernstein polynomials $b_\alpha$ are positive except of  set of zero measure, the upper bound in~\eqref{eq:fill-diagonal} yields $r_\alpha = \bar{r}$ for all~$\alpha$. 
Hence, the kernel of $S$ reads
\begin{align*}
	\ker(S) = \set{(r_{\alpha})_{\abs{\alpha} = k} \in (\mathcal{P}_k(\widehat{T}))^N \colon r_{\alpha} = r \in \mathcal{P}_k(\widehat{T}) \; \text{ for all } \alpha \text{ with } \abs{\alpha} = k}.
\end{align*}
The restriction of $S$ to $(\mathcal{P}_k(\widehat{T}))^N / \ker(S)$ is bijective on its image. We claim that the image $S(U)$ consists of all non-diagonal entries, that is, 
\begin{align}\label{eq:claim}
\mathbb{R}^{N\times N} = S(U) \oplus R \quad \text{with }R \coloneqq \lbrace (M_{\alpha,\beta})_{\alpha,\beta} \in \mathbb{R}^{N\times N}\colon M_{\alpha,\beta} = 0\text{ for }\alpha \neq \beta\rbrace.
\end{align}
We verify the claim as follows. Similar to the calculations above, we conclude for all polynomials $(r_\alpha)_{|\alpha| = k} \in (\mathcal{P}_k(\That))^N$ that
  \begin{align}
    \label{eq:goal3}
    \begin{aligned}
       \skpT{b_\alpha (r_\alpha- \bar{r})}{b_\alpha} &= \skpT{\bar{r}}{b_\alpha} - \sum_{\mu \neq \alpha} \skpT{b_\mu r_\mu}{b_\alpha} - \skpT{b_\alpha\bar{r}}{b_\alpha}\\
      &= \sum_{\mu \neq \alpha} \skpT{b_\mu(\bar{r} - r_\mu)}{b_\alpha}.
    \end{aligned}
  \end{align}
  Since the sum in the right-hand side of \eqref{eq:goal3} contains solely off-diagonal entries of $S((r_\alpha)_{|\alpha|=k})$, it vanishes by definition for all $S((r_\alpha)_{|\alpha|=k}) \in R$.
Hence, the left-hand side equals zero and so $S((r_\alpha)_{|\alpha|=k}) = 0$. This implies $S(U) \cap R = \lbrace 0 \rbrace$. The bijectivity of $S$ with respect to $U$ shows that $\dim(\mathbb{R}^{N\times N}) - \dim(S(U)) = N^2 -(N^2 - N) = N = \dim(R)$. This verifies the claim in \eqref{eq:claim}.
Thus, we can find polynomials~$(q_\alpha)_{\abs{\alpha} = k} \in (\mathcal{P}_k(\That))^N$ with 
  \begin{align}
    \label{eq:goal2}
    \skpT{b_\alpha (q_\alpha- \bar{q})}{b_\beta} = 
    - \frac{(d+k)!}{k!} \skpT{b_\alpha}{b_\beta} \qquad \text{for all $\alpha \neq \beta$}.
  \end{align}
  With the calculation in \eqref{eq:goal3}, we recover for all $\beta$ the diagonal case from~\eqref{eq:goal2} by
  \begin{align*}
    \begin{aligned}
      & \skpT{b_\beta (q_\beta- \bar{q})}{b_\beta} = \sum_{\mu \neq \beta} \skpT{b_\mu(\bar{q} - q_\mu)}{b_\beta}
      = \sum_{\mu \neq \beta} \frac{(d+k)!}{k!} \skpT{b_\mu}{b_\beta}
      \\
      &\qquad= \frac{(d+k)!}{k!} \skpT{1}{b_\mu} - \frac{(d+k)!}{k!} \skpT{b_\beta}{b_\beta}
      = 1 - \frac{(d+k)!}{k!} \skpT{b_\beta}{b_\beta}.
    \end{aligned}
  \end{align*}
This proves the existence of functions $(q_{\alpha})_{\abs{\alpha} = k }$ satisfying \eqref{eq:goal1}. 
The construction is symmetric with respect to any permutation of the barycentric coordinates which ensures~\ref{itm:local-symmetry}.
\end{proof}

\begin{remark}[Comparison with Tantardini] \label{rmk:TV-reform}
 In the linear case $k = 1$ Tantardini constructed
  in Section~7.6 of her PhD thesis \cite{Tantardini12} a Scott--Zhang type operator with weights $\psi_\ell =\abs{\That}\, \abs{\omega_\ell}^{-1} p_\ell$, for $\ell \in \vertices$, where the function $p_\ell$ reads locally
  \begin{align*}
    \frac{p_\ell}{(d+1)!}
    &=  \lambda_\ell \Bigg(
      \lambda_\ell^2 
      - \frac{(d+2)(d+5)}{2} \sum_{j \neq \ell} \lambda_j^2
      + \frac{d^2+7d+16}{2}\lambda_\ell \sum_{j \neq \ell} \lambda_j
      + 2\,\sum_{\substack{i,j \\ i < j,i \neq \ell, j \neq \ell}} \lambda_i \lambda_j \bigg).
  \end{align*}
  Since $
    1 = 
    \sum_\ell \lambda_\ell^2 +  \sum_{j, \ell \colon j<\ell} 2 \lambda_j \lambda_\ell$,
  this can be rewritten as
  \begin{align*}
    \frac{p_\ell}{(d+1)!}
    &= \lambda_\ell \Bigg( 1 
      - \Big(\frac{(d+2)(d+5)}{2} +1\Big) \sum_{j \neq \ell} \lambda_j^2
      + \Big(\frac{d^2+7d+16}{2}-2\Big)\lambda_\ell \sum_{j \neq \ell} \lambda_j
      \Bigg)
    \\
    &= \lambda_\ell \Bigg( 1 
      - \frac{(d+3)(d+4)}{2} \sum_{j \neq \ell} \lambda_j^2
      + \frac{(d+3)(d+4)}{2}\lambda_\ell \sum_{j \neq \ell} \lambda_j
      \Bigg)
    \\
    &= \lambda_\ell \Bigg( 1 
      - \frac{(d+3)(d+4)}{2} \sum_j \lambda_j^2
      + \frac{(d+3)(d+4)}{2}\lambda_\ell \sum_j \lambda_j
      \Bigg)
    \\
    &= \lambda_\ell \Bigg( 1
      + \frac{(d+3)(d+4)}{2}\Big( \lambda_\ell - \sum_j \lambda_j^2 \Big)
      \bigg).
  \end{align*}
  We obtain~\eqref{eq:ansatz-p} with $c_1^{-1} = (d+1)!$ and
  \begin{align*}
    q_{e_\ell} & = \frac{(d+4)!}{2(d+2)} \lambda_\ell.
  \end{align*}
This representation of Tantardini's operator for~$k=1$ inspired our ansatz in~\eqref{eq:ansatz-p}.
\end{remark}

\subsection{Global biorthogonal basis}\label{sec:global-basis}
In the following we introduce a system of functions that is biorthogonal to Bernstein polynomials. It can be viewed as a basis of the spanned subspace of $\mathcal{L}^1_{3k}(\tria)$, which is dual to the Bernstein basis of $\mathcal{L}^1_{k}(\tria)$.

\begin{proposition}[Global dual basis]\label{pro:localWeights}
  There exist polynomials $\psi_i \in \mathcal{L}^1_{3k}(\tria)$ for all Lagrange nodes $i\in \nodes$ with the following properties.
  \begin{enumerate}
  \item \label{itm:product-structure}\textbf{Product structure.} For each $\psi_i$ there exists a  $\zeta_i \in \mathcal{L}^1_{2k}(\omega_i)$ such that
    \begin{align*}
      \psi_i = b_i \zeta_i.
    \end{align*}
    In particular the support is local in the sense that $\support \psi_i \subset \support b_i$.
  \item\label{itm:dualBasis} \textbf{Biorthogonality.} We have for all nodes $i,j\in \nodes$ the relation 
    \begin{align*}
      \skp{b_i}{\psi_j}_\Omega = \delta_{i,j}.
    \end{align*}
  \item \label{itm:Boundedness} \textbf{Boundedness.} We have the upper bounds 
    \begin{align*}
      \norm{ \psi_i }_{L^{\infty}(\omega_i)} \lesssim \abs{\omega_i}^{-1} ,\quad \norm{ \nabla \psi_i }_{L^{\infty}(\omega_i)} \lesssim \textup{diam}(\omega_i)^{-1} \abs{\omega_i}^{-1}.
    \end{align*}
    The hidden constant in the first estimate depends on $k$ and the one in the second estimate additionally on the shape regularity of $\tria$.
  \item\label{itm:preserv-mass} \textbf{Preservation of mass.} We have the identity
    \begin{align*}
      \sum_{i \in \nodes}\langle 1,b_i\rangle_\Omega \,\psi_i = 1.
    \end{align*}
  \end{enumerate}
\end{proposition}
\begin{proof}
  Let us consider an arbitrary but fixed $T \in \tria$. 
  The restriction of each global $b_i$ with $i \in \nodes \cap T$ to $T$ is given by $b_{\alpha(i,T)}$, where $\alpha(i,T) \in \mathbb{N}_0^d$ with $\abs{\alpha} = k$, see~\eqref{eq:def-Bernstein}. 
  Let $z_\alpha \in \mathcal{P}_{2k}(T)$ and $p_\alpha \in \mathcal{P}_{3k}(T)$ with $p_\alpha = b_\alpha z_\alpha$ denote the local basis functions given by Lemma~\ref{lem:local-basis} transformed from~$\That$ to~$T$.  
  Now, we can locally patch these functions together to obtain global functions $z_i \in \mathcal{L}^1_{2k}(\tria(\omega_i))$ and $p_i \in \mathcal{L}^1_{3k}(\tria(\omega_i))$ with $\tria(\omega_i)\coloneqq \lbrace T'\in \tria \colon T \subset \omega_i\rbrace$ such that $z_i = z_{\alpha(i,T)}$ and $p_i = p_{\alpha(i,T)}$ on each~$T\in \tria(\omega_i)$. 
  The symmetry of the local functions with respect to permutations of the barycentric coordinates due to Lemma~\ref{lem:local-basis} ensures the continuity of the functions on $\omega_i$ in this construction.

  Using the local properties in Lemma~\ref{lem:local-basis} we proceed to prove the properties of the functions $p_i$ and $z_i$. 
  By Lemma~\ref{lem:local-basis}\,\ref{itm:local-product} we have that $p_i =b_iz_i$ and $\support(p_i) \subset \support (b_i) = \omega_i$. 
This allows us to extend $p_i$ by zero to any $T \notin \tria(\omega_i)$ and we obtain that $p_i \in \mathcal{L}^1_{3k}(\tria)$. 
  Using the affine transformation and Lemma~\ref{lem:local-basis}\,\ref{itm:local-duality} shows
    \begin{align}\label{eq:local-dual}
      \skp{b_i}{p_j}_T = \frac{\abs{T}}{\abs{\That}} \delta_{i,j}\qquad\text{for all }i,j  \in \nodes\text{ and }T\in \tria.
    \end{align}
   Lemma~\ref{lem:local-basis}\,\ref{itm:local-one} yields that
    \begin{align}\label{eq:sum-pi}
     \sum_{i \in \nodes }  p_i =    \frac{1}{c_k}.
    \end{align}
  We define the rescaled functions $\psi_i \in \mathcal{L}^1_{3k}(\tria)$ and $\zeta_i \in \mathcal{L}^1_{2k}(\tria(\omega_i))$ by
  \begin{align*}
    \psi_i &\coloneqq \frac{\abs{\That}}{\abs{\omega_i}} p_i \qquad \text{and} \qquad
    \zeta_i \coloneqq \frac{\abs{\That}}{\abs{\omega_i}} z_i.
  \end{align*}
 We have $\psi_i =b_i\zeta_i$ and $\support(\psi_i) \subset \support (b_i) = \omega_i$, which proves~\ref{itm:product-structure}. 
 Furthermore, by~\eqref{eq:local-dual} for all $i,j \in \nodes$ we have that 
 \begin{align*}
   \skp{b_i}{\psi_j}_\Omega = \frac{\abs{\hat{T}}}{\abs{\omega_j}}\skp{b_i}{p_j}_\Omega = \frac{\abs{\hat{T}}}{\abs{\omega_j}}\, \sum_{T \in \tria(\omega_i \cap \omega_j)} \frac{\abs{T}}{\abs{\That}} \delta_{i,j} = \delta_{i,j}.
 \end{align*}
 This proves~\ref{itm:dualBasis}. 
 Further,~\eqref{eq:sum-pi} yields that
 \begin{align*}
   \sum_{i \in \nodes} \skp{1}{b_i}_\Omega\, \psi_i = \sum_{i \in \nodes} \frac{\abs{\hat{T}}}{\abs{\omega_i}} \skp{1}{b_i}_\Omega\, p_i =  c_k \sum_{i \in \nodes}  p_i =  1,
 \end{align*}
 which proves~\ref{itm:preserv-mass}. 
 Finally,  a scaling argument yields the upper bounds in~\ref{itm:Boundedness}.
\end{proof}

\subsection{Construction and proof of main results}\label{subsec:ProofThm1}
With the biorthogonal system of Bernstein polynomials $(b_i)_{i \in \nodes}$ and functions $(\psi_i)_{i \in \nodes}$ from Proposition~\ref{pro:localWeights} at hand we are in the position to design the Scott--Zhang type interpolation operator $\Pizero$.
The construction of~$\Pi$ requires an adaptation close to the boundary and we postpone this to Section~\ref{sec:BoundaryCorr}. 
Recall that $\nodes$ denotes the set of all Lagrange nodes of $\mathcal{L}^1_{k}(\tria)$ and denote by $\nodesint \coloneqq \nodes \setminus \partial \Omega$ the set of interior Lagrange nodes.  
We define the projection operator $\Pizero\colon W^{-1,2}(\Omega) \to \mathcal{L}^1_{k,0}(\tria)$ by
\begin{align}
  \label{eq:def-Pi}
  \Pizero v &\coloneqq \sum_{i \in \nodesint} \skp{v}{\psi_i}_\Omega \,b_i\qquad\text{for }v \in L^2(\Omega).
\end{align}
Thanks to the fact that $\psi_i \in W^{1,\infty}_0(\Omega)$, this operator is defined on the dual space $(W^{1,\infty}_0(\Omega))^*$. 
In particular, $\Pizero \xi$ is defined for any $\xi \in W^{-1,p}(\Omega)$ with $p \in (1,\infty)$. 
The $L^2$-adjoint operator $\Pizero^*\colon L^2(\Omega) \to \textup{span}\lbrace \psi_i\colon i\in \nodesint\rbrace \subset \mathcal{L}^1_{3k,0}(\tria)$ reads
\begin{align}
  \label{eq:def-Pi-dual}
  \Pizero^* v = \sum_{i\in \nodesint} \skp{v}{b_i }_\Omega\, \psi_i\qquad\text{for all }v \in L^2(\Omega). 
\end{align}
Since $b_i \in W^{1,\infty}_0(\Omega)$, this operator is also defined on  $W^{-1,p}(\Omega)$ for any $p\in(1,\infty)$.

\begin{lemma}[Non-negative orders]
  \label{lem:positive-norms}
 The operator $\Pizero$ is a projection. Moreover, $\Pizero$ and $\Pizero^*$ are locally stable in the sense that for all  $p\in [1,\infty]$, for any $0 \leq s \leq k$ and any $T\in \tria$ we have that 
  \begin{align}
    \label{eq:positive-norms1}
    \begin{alignedat}{2}
      \lVert \Pizero v \rVert_{L^p(T)}& \lesssim \lVert v \rVert_{L^p(\omega_T)} &\qquad&\text{for all }v\in L^p(\Omega),
      \\
      \lVert \nabla^s\Pizero w \rVert_{L^p(T)}& \lesssim \lVert \nabla^s w \rVert_{L^p(\omega_T)}&&\text{for all }w\in W^{1,p}_0(\Omega)\cap W^{s,p}(\Omega),
      \\
      \lVert \Pizero^* v \rVert_{L^p(T)}& \lesssim \lVert v \rVert_{L^p(\omega_T)}&\qquad&\text{for all }v\in L^p(\Omega),\\
      \lVert  \nabla \Pizero^* w \rVert_{L^p(T)}& \lesssim \lVert \nabla w \rVert_{L^p(\omega_T)} &&\text{for all }w\in W^{1,p}_0(\Omega).
    \end{alignedat}
  \end{align}
  In addition, $\Pizero$ and $\Pizero^*$ have the following local approximation properties for all $p \in[1,\infty]$, all $T \in \tria$, and all $w\in W_0^{1,p}(\Omega)\cap W^{s,p}(\Omega)$:
  \begin{align}
    \label{eq:positive-norms2}
    \begin{alignedat}{2}
      \norm{\nabla^m (w - \Pizero w)}_{L^p(T)} &\lesssim h_T^{s-m} \norm{ \nabla^s w }_{L^p(\omega_T)}&\quad&\text{for all $0 \leq m \leq s \leq k+1$},
      \\
      \lVert \nabla^m(w - \Pizero^*w )\rVert_{L^p(T)}& \lesssim h_T^{1-m} \lVert \nabla w \rVert_{L^p(\omega_T)}  &\quad&\text{for $m=0,1$.}
    \end{alignedat}
  \end{align}
  The adjoint operator preserves constants on interior simplices, that is,
  \begin{align}
    \label{eq:positive-norms3}
    (\Pizero^*1)|_T = 1 \qquad \text{ for all }\, T\in \tria \text{ with }T\cap \partial \Omega = \emptyset.
  \end{align}
\end{lemma}
\begin{proof}
  Due to the biorthogonality of $(b_i)_{i \in \nodes}$ and $(\psi_i)_{i \in \nodes}$ by Proposition~\ref{pro:localWeights} the operators~$\Pizero$ and~$\Pizero^*$ are linear projections.
  By Proposition~\ref{pro:localWeights}\,\ref{itm:preserv-mass} we have
  \begin{align*}
    \Pizero^* 1 &= \sum_{i \in \nodesint} \skp{1}{b_i}_\Omega\, \psi_i = 1 - \sum_{i \in \nodes \setminus\nodesint} \skp{1}{b_i}_\Omega\, \psi_i.
  \end{align*}
  In particular, this shows~\eqref{eq:positive-norms3}.   
  From Proposition~\ref{pro:localWeights}\,\ref{itm:Boundedness} we conclude
  \begin{align}
    \label{eq:1}
    \norm{\psi_i}_{L^{p'}(\omega_i)}  \norm{b_i}_{L^p(\omega_i)} \lesssim 1\qquad\text{for all }p \in [1,\infty].
  \end{align}
  This implies the local $L^p$-stability of~$\Pizero$ for $p\in[1,\infty)$ as
  \begin{align*}
    \norm{ \Pizero v }_{L^p(T)}
    &\lesssim \bigg(\sum_{i \in \nodes \cap T} \big(\norm{v}_{L^p(\omega_i)} \norm{\psi_i}_{L^{p'}(\omega_i)}  \norm{b_i}_{L^p(T)} \big)^p \bigg)^{\frac 1p}
    \lesssim \norm{v}_{L^p(\omega_T)},
  \end{align*}
and for $p=\infty$ analogously. 
  The same proof applies to the local $L^p$-stability of~$\Pizero^*$. The projection property of~$\Pizero$ on~$\mathcal{L}^1_{k,0}(\tria)$ and inverse estimates yield
  \begin{align*}
    \norm{\nabla^m(w-\Pizero w)}_{L^p(T)}
    &\lesssim \inf_{w_h \in \mathcal{L}^1_{k,0}(\tria) } \big( 
      \norm{\nabla^m(w-w_h)}_{L^p(T)} + h_T^{-m}
      \norm{\Pizero(w-w_h)}_{L^p(T)} \big)
    \\
    &\lesssim \inf_{w_h \in \mathcal{L}^1_{k,0}(\tria) } \big( 
      \norm{\nabla^m(w-w_h)}_{L^p(T)} + h_T^{-m}
      \norm{w-w_h}_{L^p(\omega_T)} \big)
    \\
    &\lesssim h_T^{s-m} \norm{ \nabla^s w }_{L^p(\omega_T)}\qquad\text{for all }0 \leq m \leq s \leq k+1.
  \end{align*}
  This proves the local approximation estimate of~$\Pizero$ in~\eqref{eq:positive-norms2} and, taking $m = s$, the local stability in~\eqref{eq:positive-norms1}. 
  We continue to consider~$\Pizero^*$. 
  We define $c_T \coloneqq \abs{\omega_T}^{-1} \int_{\omega_T} w \dx$ for $T \cap \partial \Omega = \emptyset$ and $c_T \coloneqq 0$ for $T \cap \partial \Omega \neq \emptyset$. 
  Since $\Pizero^* c_T = c_T$ in both cases, inverse estimates and Friedrichs'/\Poincare's inequality imply that
  \begin{align*}
    \norm{\nabla^m(w-\Pizero^* w)}_{L^p(T)}
    &\lesssim 
      \norm{\nabla^m(w-c_T)}_{L^p(T)} + h_T^{-m}
      \norm{\Pizero^*(w-c_T)}_{L^p(T)}
    \\
    &\lesssim \norm{\nabla^m( w-c_T)}_{L^p(T)} + h_T^{-m}
      \norm{w-c_T}_{L^p(\omega_T)}
    \\
    &\lesssim h_T^{1-m} \norm{ \nabla w }_{L^p(\omega_T)}.
  \end{align*}
  This proves the local approximation property of~$\Pizero^*$ in~\eqref{eq:positive-norms2}. 
  Taking $m = 1$ also the local Sobolev stability in~\eqref{eq:positive-norms1} follows, which finishes the proof. 
\end{proof}
\begin{remark}[Preservation of constants by the adjoint]
  To conclude optimal convergence rates in dual Sobolev spaces we have to take advantage of approximation properties of the adjoint operator. 
  Such properties would automatically be satisfied by self-adjoint projection operators. 
  However, in Section~\ref{sec:SelfAdjProj} below we argue that our operator cannot be self-adjoint.    
  The weaker condition~\eqref{eq:positive-norms3} cures this lack of self-adjointness. 
  This weaker condition has been introduced for the lowest order case in \cite[Prop.~7.22]{Tantardini12} and has also been employed by 
Tantardini and Veeser in \cite[Sec.~5]{TantardiniVeeser16}. 
  Therein, they introduce  for all polynomial degrees~$k \in \setN$ a local projection with optimal approximability in~$W^{-1,2}(\Omega)$  for functions in $L^2(\Omega)$. 
  However, they use discontinuous weight functions and consequently their operator does not extend to~$W^{-1,2}(\Omega)$-functions.

\end{remark}

\begin{remark}[Preservation of mass]\label{rem:PresMass}
Let the operator $P$ be defined as our operator $\Pizero$ but including the boundary nodes, i.e., 
$P v \coloneqq \sum_{i \in \nodes} \skp{v}{\psi_i}_{\Omega} \, b_i.$
Proposition~\ref{pro:localWeights}\,\ref{itm:preserv-mass} ensures that $\skp{P v}{1}_{\Omega} = \skp{v}{1}_{\Omega}$, and in this sense $P$ preserves the mass.  
By duality this property is equivalent to $P^* 1 = 1$ on $\Omega$. 
\end{remark}

The preservation of constants by the adjoint operator $\Pi_0^*$ in \eqref{eq:positive-norms3} is key in the approximation results of $\Pi_0^*$ in Sobolev norms, see \eqref{eq:positive-norms2}. Consequently, by duality it is the crucial ingredient in the proof of approximation properties of $\Pi_0$ in negative Sobolev norms.

\begin{lemma}[Localization of norms]
  \label{lem:localized-norms}
  For all $p \in (1,\infty)$ and $\xi \in W^{-1,p}(\Omega)$ we have
  \begin{align*}
    \norm{\xi - \Pizero \xi }_{W^{-1,p}(\Omega)} & \eqsim \Big(\sum_{j \in \mathcal{V}} \norm{\xi- \Pizero \xi}_{W^{-1,p}(\omega_j)}^p \Big)^{1/p},
    \\
    \lVert \xi - \Pizero^* \xi \rVert_{W^{-1,p}(\Omega)} &\eqsim \Big(\sum_{j\in \mathcal{V}} \lVert \xi - \Pizero^* \xi  \rVert_{W^{-1,p}(\omega_j)}^p\Big)^{1/p}.
  \end{align*}
\end{lemma}
\begin{proof}
  Let $(\phi_j)_{j \in \vertices}$ denote the nodal basis of~$\mathcal{L}^1_1(\tria)$, which forms a partition of unity.
  For all $w \in W^{1,p'}_0(\Omega)$ we use the fact that $\phi_j(w - \Pizero^*w) \in W^{1,p'}_0(\omega_j)$, the approximation property in~\eqref{eq:positive-norms2} and the finite overlap of the patches~$\omega^2_j$ to calculate
  \begin{align*}
    \skp{\xi - \Pizero \xi}{w}_\Omega
    &=
      \skp{\xi - \Pizero \xi}{w - \Pizero^*w}_\Omega
    \\
    &=  \sum_{j \in \vertices} \skp{\xi - \Pizero \xi}{\phi_j(w - \Pizero^*w)}_\Omega
    \\
    &\leq  \sum_{j\in \mathcal{V}} \Vert \xi - \Pizero \xi \rVert_{W^{-1,p}(\omega_j)} \lVert \nabla( \phi_j( w  - \Pizero^* w)) \rVert_{L^{p'}(\omega_j)}
    \\
    &\lesssim \sum_{j\in \mathcal{V}} \Vert \xi - \Pizero \xi \rVert_{W^{-1,p}(\omega_j)} \lVert \nabla w \rVert_{L^{p'}(\omega^2_j)}
    \\
    &\lesssim \Big( \sum_{j\in \mathcal{V}}\Vert \xi - \Pizero \xi \rVert_{W^{-1,p}(\omega_j)}^p\Big)^{1/p} \norm{\nabla w}_{L^{p'}(\Omega)}.
  \end{align*}
This proves
  \begin{align}
    \label{eq:2}
    \norm{\xi - \Pizero \xi}_{W^{-1,p}(\Omega)}
    &\lesssim \Big( \sum_{j\in \mathcal{V}}\norm{ \xi - \Pizero \xi }_{W^{-1,p}(\omega_j)}^p\Big)^{1/p}.
  \end{align}
  The reverse estimate does not rely on the properties of the projection. 
  Indeed, for any $g \in W^{-1,p}(\Omega)$ we find a sequence $(s_j)_j \in \ell^{p'}(\vertices)$ with $\norm{(s_j)_j}_{\ell^{p'}(\vertices)}=1$ and $w_j \in W^{1,p'}_0(\omega_j)$ with $\norm{\nabla w_j}_{L^{p'}(\omega_j)}=1$ such that
  \begin{align*}
    \Big( \sum_{j\in \mathcal{V}}\norm{ g}_{W^{-1,p}(\omega_j)}^p\Big)^{1/p}
    &= \sum_{j\in \mathcal{V}} \norm{ g }_{W^{-1,p}(\omega_j)} s_j
    = \sum_{j\in \mathcal{V}} \skp{ g}{w_j}_{\omega_j}\, s_j.
  \end{align*}
%
  By the fact that $\sum_{j\in \vertices} s_j w_j \in W^{1,p'}_0(\Omega)$ we obtain
  \begin{align*}
    \begin{aligned}      
      &
      \Big( \sum_{j\in \mathcal{V}}\norm{g }_{W^{-1,p}(\omega_j)}^p\Big)^{1/p}
      \leq \norm{g}_{W^{-1,p}(\Omega)} \biggnorm{ \sum_{j \in \vertices} s_j\nabla w_j}_{L^{p'}(\Omega)}.
    \end{aligned}
  \end{align*}
 Since each simplex $T\in \tria$ is covered by $d+1$ nodal patches, we have
 \begin{align*}
 \biggnorm{ \sum_{j \in \vertices} s_j\nabla w_j}_{L^{p'}(\Omega)}^{p'}& = \sum_{T\in \tria} \int_T \Big\lvert \sum_{j\in \vertices} s_j \nabla w_j\Big\rvert^{p'} \dx \leq (d+1)^{p'/p} \sum_{T\in \tria} \sum_{j\in \vertices} \int_T  |s_j \nabla w_j|^{p'}\dx\\
 &= (d+1)^{p'/p} \sum_{j\in \vertices} |s_j|^{p'} \lVert \nabla w_j \rVert_{L^{p'}(\Omega)}^{p'} = (d+1)^{p'/p}.
 \end{align*}
 Combining these estimates shows 
 \begin{align}\label{eq:localization-trivial}
   \Big( \sum_{j\in \mathcal{V}}\norm{g }_{W^{-1,p}(\omega_j)}^p\Big)^{1/p} \leq (d+1)^{1/p} \lVert g \rVert_{W^{-1,p}(\Omega)}.
 \end{align} 
The claim for~$\Pizero$ follows by~\eqref{eq:2} and choosing $g= \xi - \Pizero \xi$ in \eqref{eq:localization-trivial}. 
The statement for~$\Pizero^*$ follows by exchanging the roles of~$\Pizero$ and $\Pizero^*$ in the proof.
\end{proof}
Related localization techniques for negative Sobolev spaces $W^{-1,p}(\Omega)$ applied to residuals in a~posteriori analysis are employed in~\cite{CDN.2012,BMV.2020,KreuzerVeeser21}. 
\begin{lemma}[Negative order]
  \label{lem:negative}
  Let $p \in (1,\infty)$, $j\in \mathcal{V}$, and $0 \leq s \leq k+1$. 
  Then
  \begin{align}
    \label{eq:negative1}
    \begin{alignedat}{2}
      \norm{\xi - \Pizero \xi }_{W^{-1,p}(\omega_j)} & \lesssim   \norm{\xi}_{W^{-1,p}(\omega_j^2)}  &&\text{ for all } \xi \in W^{-1,p}(\Omega),\\
      \norm{v - \Pizero v }_{W^{-1,p}(\omega_j)} & \lesssim   h_j \norm{v}_{L^{p}(\omega_j^2)}  &&\text{ for all } v \in L^{p}(\Omega),\\
      \norm{w - \Pizero w }_{W^{-1,p}(\omega_j)}&\lesssim  h_j^{s+1} \norm{\nabla^s w}_{L^p(\omega^2_j)} 
      \; &&\text{ for all  }w\in W^{1,p}_0(\Omega)\cap W^{s,p}(\Omega),
      \\
      \norm{\xi - \Pizero^* \xi }_{W^{-1,p}(\omega_j)} & \lesssim   \norm{\xi}_{W^{-1,p}(\omega_j^2)}  &\quad&\text{ for all } \xi \in W^{-1,p}(\Omega),
      \\
      \lVert  v- \Pizero^* v \rVert_{W^{-1,p}(\omega_j)} & \lesssim h_j \lVert v \rVert_{L^p(\omega_j^2)} 
      &&\text{ for all }v \in L^{p}(\Omega),\\
      \lVert w - \Pizero^* w \rVert_{W^{-1,p}(\omega_j)} & \lesssim h_j^2 \lVert \nabla w \rVert_{L^p(\omega^2_j)}
      &&\text{ for all }  w\in W^{1,p}_0(\Omega).
    \end{alignedat}
  \end{align}
  Furthermore, for all $\xi \in W^{-1,p}(\Omega)$ one has
  \begin{align}
    \label{eq:negative2}
    \begin{aligned}
      \norm{\Pizero \xi}_{W^{-1,p}(\Omega)} 
      \lesssim \norm{\xi}_{W^{-1,p}(\Omega)}
      \quad \text{and}\quad
      \norm{\Pizero^* \xi}_{W^{-1,p}(\Omega)} 
      \lesssim \norm{\xi}_{W^{-1,p}(\Omega)}.
    \end{aligned}
  \end{align}
\end{lemma}
\begin{proof}
  Let $j\in \mathcal{V}$ and $\xi \in W^{-1,p}(\Omega)$. 
By the locality of~$\Pizero$ and $\Pizero^*$ and~\eqref{eq:positive-norms2} (with $m=0$) for any $g \in W^{1,p'}_0(\omega_j)$, and noting that $g - \Pizero^*g \in W^{1,p'}_0(\omega_j^2)$, we obtain
  \begin{align*}
    \skp{\xi - \Pizero \xi}{g}_{\omega_j}
    &= \skp{\xi }{g- \Pizero^* g}_{\Omega} =
      \skp{\xi}{g - \Pizero^* g}_{\omega_j^2}     
    \\
    & \leq  \norm{\xi}_{W^{-1,p}(\omega_j^2)} \norm{\nabla (g - \Pizero^* g)}_{L^{p'}(\omega_j^2)}\lesssim \norm{\xi}_{W^{-1,p}(\omega^2_j)}
      \norm{\nabla g}_{L^{p'}(\omega_j)}.
  \end{align*}
  Taking the supremum over all~$g \in W^{1,p'}_0(\omega_j)$ with $\norm{\nabla g}_{L^{p'}(\omega_j)}\leq 1$ proves the first statement in~\eqref{eq:negative1}.

  As a consequence of \Poincare{}'s inequality we have $\norm{g}_{L^{p'}(\omega_j)} \lesssim h_j \norm{\nabla g}_{L^{p'}(\omega_j)}$  for all $g \in W^{1,p'}_0(\omega_j)$. By duality we obtain
  \begin{align}
    \label{eq:negative-poincare}
    \norm{v}_{W^{-1,p}(\omega_j)} \lesssim h_j \norm{v}_{L^p(\omega_j)} \quad \text{for all $v \in L^{p}(\omega_j)$.}
  \end{align}
The same holds for~$\omega_j$ replaced by~$\omega_j^2$, so the second claim of~\eqref{eq:negative1} follows from the first one. 
  Applying~\eqref{eq:negative-poincare} with $v \coloneqq w - \Pizero w$ for $w \in W^{1,p}_0(\Omega)$ we obtain with~\eqref{eq:positive-norms2} for $m = 0$ and $s=0,\dots,k+1$ that
  \begin{align*}
    \norm{w - \Pizero w}_{W^{-1,p}(\omega_j)}
    &\lesssim
      h_j\norm{w - \Pizero w}_{L^p(\omega_j)}
      \lesssim  h_j^{s+1} \norm{\nabla^s w}_{L^p(\omega^2_j)} .
  \end{align*}
  This proves the estimates for~$\Pizero$ in~\eqref{eq:negative1}. 
  The ones for~$\Pizero^*$ follow in the same manner.

  The global stability in~\eqref{eq:negative2} follows from the local one by the localization of the norm in Lemma~\ref{lem:localized-norms}. 
  Applying those and the arguments used in~\eqref{eq:localization-trivial} leads to
  \begin{align*}
    \norm{\Pizero \xi}_{W^{-1,p}(\Omega)}
    &\leq
    \norm{\xi - \Pizero \xi}_{W^{-1,p}(\Omega)} + 
      \norm{\xi}_{W^{-1,p}(\Omega)}
    \\
    &\lesssim 
      \Big(\sum_{j \in \mathcal{V}} \norm{\xi- \Pizero \xi}_{W^{-1,p}(\omega_j)}^p \Big)^{1/p}+ 
      \norm{\xi}_{W^{-1,p}(\Omega)}
    \\
    &\lesssim 
      \Big(\sum_{j \in \mathcal{V}} \norm{\xi}_{W^{-1,p}(\omega_j^2)}^p \Big)^{1/p}+ 
      \norm{\xi}_{W^{-1,p}(\Omega)}
    \\
    &\lesssim  \norm{\xi}_{W^{-1,p}(\Omega)}.
  \end{align*}
  The same proof works for~$\Pizero^*$.
\end{proof}

Now Lemmas~\ref{lem:positive-norms}, \ref{lem:localized-norms} and \ref{lem:negative} prove Theorems~\ref{thm:interpolOperator} and \ref{thm:PropPi*}. 
\begin{remark}[AFEM]
  Let $\widehat{\tria}$ be a refinement of $\tria$. Then the local design of $\Pizero\colon L^2(\Omega) \to  \mathcal{L}^1_{k,0}(\tria)$ shows that for all $\hat{v}_h\in \mathcal{L}^1_{k,0}(\widehat{\tria})$ one has 
  \begin{align}\label{eq:identAFEM}
    \Pizero \hat{v}_h = \hat{v}_h\text{ on }T \qquad\text{for all }T\in \widehat{\tria}\cap \tria \text{with } \omega_T \subset \widehat{\tria}\cap\tria.
  \end{align}
  In other words, the restriction $\Pi_0\colon \mathcal{L}^1_{k,0}(\widehat{\tria}) \to \mathcal{L}^1_{k,0}(\tria)$ is the identity on non-refined elements~$T$ that are surrounded by non-refined elements.
  This is vital in the proof of discrete reliability for adaptive finite element schemes; cf.~\cite{CarstensenFeischlPagePraetorius14,CarstensenRabus17,Stevenson07}.
  In certain situations it is possible \cite[Lem.~3.6]{CasconKreuzerNochettoSiebert08} or needed \cite[Lem.~4.3]{DieningKreuzerStevenson16} to have equality in~\eqref{eq:identAFEM} on all unrefined simplices $T\in \widehat{\tria} \cap \tria$. 
  The following example shows however that this requirement is not compatible with the key property $\Pizero^*1= 1$ locally as in~\eqref{eq:positive-norms3} needed for approximation results in~$W^{-1,p}(\Omega)$ as in Theorem~\ref{thm:interpolOperator}.  

  Let $\tria$ be a triangulation of $\Omega = (0,1)$ and let $\widehat{\tria}$ be a refinement that results from the bisection of a single interval $T\in \tria$. 
  Let $\hat{\phi} \in \mathcal{L}^1_1(\widehat{\tria})$ denote the hat function related to the new node $\ell = \textup{mid}(T)$, that is, $\hat{\phi}(i) = 0$ for all nodes $i\in \nodes$ and $\hat{\phi}(\ell) = 1$. 
  Then any operator $\tilde{\Pi}\colon W^{1,p}_0(\Omega) \to \mathcal{L}^1_{1,0}(\tria)$ with
  $\tilde{\Pi} \hat{\varphi} = \hat{\varphi}$ on all $T\in \widehat{\tria} \cap \tria$ must satisfy $ \tilde{\Pi} \hat{\varphi} = 0$. 
  This yields $0 = \langle \tilde{\Pi} \hat{\varphi} , 1 \rangle_\Omega = \langle  \hat{\varphi} , \tilde{\Pi}^* 1 \rangle_\Omega$ and thus shows  that $\tilde{\Pi}^* 1  \neq 1$.
\end{remark}

\subsection{\texorpdfstring{Boundary correction for~$\Pi$}{Boundary correction}}
\label{sec:BoundaryCorr}

Since $\Pizero$ maps to~$\mathcal{L}^1_{k,0}(\tria)$ the stability estimate $\norm{\nabla \Pizero 1}_{L^2(\Omega)} \lesssim \norm{\nabla 1}_{L^2(\Omega)} =0$ cannot hold and requires a modification of our operator. 
In this subsection we show how to obtain the stability and approximation estimates of Theorem~\ref{thm:interpolOperator} for functions without zero trace. 

Let $(\phi_j)_{j \in \vertices}$ denote the nodal basis of~$\mathcal{L}^1_1(\tria)$ and define $\eta \coloneqq \sum_{j \in \verticesint} \phi_j\in \mathcal{L}^1_{1,0}(\tria)$, then $\eta$ is one on all interior elements and zero on the boundary. 
The Riesz representation theorem yields that for any boundary node $i\in \mathcal{N}\setminus \mathcal{N}^\circ$ there exists a function $\rho_i \in \textup{span}\lbrace b_\ell \colon \ell\in \mathcal{N}\cap \omega_i\rbrace \subset \mathcal{L}^1_k(\tria)$ with
\begin{align}
  \langle b_i \eta \rho_i, b_\ell\rangle_\Omega = \langle \psi_i, b_\ell\rangle_\Omega \qquad \text{for all }\ell \in \nodes \cap \omega_i.
\end{align}
This allows us to replace any function~$\psi_i$, which is non-zero on the boundary, by the equivalent weight~$b_i\eta \rho_i\in \mathcal{L}^1_{2k+1,0}(\tria)$. 
Hence, we define 
\begin{align}
  \widetilde{\psi}_i \coloneqq \begin{cases}
    \psi_i&\text{for all }i \in \nodes^\circ \coloneqq \nodes \setminus \partial \Omega,\\
    b_i \eta \rho_i&\text{for all }i \in \nodes \setminus \nodes^\circ.
  \end{cases}
\end{align}
\begin{proposition}[Modified dual basis]\label{pro:localWeightsMod}
  The polynomials $\tpsi_i \in \mathcal{L}^1_{3k,0}(\tria)$ for $i \in \nodes$ have the following properties.
  \begin{enumerate}
  \item \label{itm:product-structureMod}\textbf{Product structure.} 
  For each $\tpsi_i$ there exists a  $\widetilde{\zeta}_i \in \mathcal{L}^1_{2k}(\omega_i)$ such that
    \begin{align*}
      \tpsi_i = b_i \widetilde{\zeta}_i.
    \end{align*}
    In particular the support is local in the sense that $\support \tpsi_i \subset \support b_i$. 
  \item\label{itm:dualBasisMod} \textbf{Biorthogonality.} We have for all nodes $i,j \in \nodes$ the relation 
    \begin{align*}
      \skp{b_i}{\tpsi_j}_\Omega = \delta_{i,j}.
    \end{align*}
  \item \label{itm:BoundednessMod} \textbf{Boundedness.} We have the upper bounds 
    \begin{align*}
      \norm{ \tpsi_i }_{L^{\infty}(\omega_i)} \lesssim \abs{\omega_i}^{-1} ,\quad \norm{ \nabla \tpsi_i }_{L^{\infty}(\omega_i)} \lesssim \textup{diam}(\omega_i)^{-1} \abs{\omega_i}^{-1}.
    \end{align*}
    The hidden constants depend on $k$ and the shape regularity of $\tria$.
  \item\label{itm:preserv-massMod} \textbf{Preservation of mass on interior elements.} We have the identity
    \begin{align*}
      \sum_{i \in \nodes}\langle 1,b_i\rangle_\Omega \, \tpsi_i|_T = 1\qquad\text{for all }T\in \tria\text{ with }T \cap \partial \Omega = \emptyset.
    \end{align*}
  \end{enumerate}
\end{proposition}
\begin{proof}
  This result follows by Proposition~\ref{pro:localWeights} and scaling arguments.
\end{proof}
We define the modification of~$\Pizero$ as $\tPi\colon L^2(\Omega)\to \mathcal{L}^1_{k}(\Omega)$ with
\begin{align}\label{eq:ScottZhangMod}
  \tPi v \coloneqq \sum_{i\in\nodes} \langle v,\tpsi_i\rangle_\Omega \,b_i\qquad\text{for all }v \in L^2(\Omega).
\end{align}
In contrast to the definition of~$\Pizero$ the sum includes boundary nodes. 
Again, since $\tpsi_i \in W^{1,\infty}_0(\Omega)$, the operator $\Pi$ extends to an operator on~$W^{-1,p}(\Omega)$ with $p\in (1,\infty)$.
The {$L^2$-adjoint} $\tPi^*\colon L^2(\Omega) \to \mathcal{L}^1_{3k,0}(\tria)$ reads
\begin{align}
  \tPi^* v = \sum_{i\in\nodes} \langle v,b_i\rangle_\Omega \, \tpsi_i\qquad\text{for all }v \in L^{p}(\Omega).
\end{align}
Since the Bernstein polynomials~$b_i$ for $i \in \partial \Omega$ do not have zero boundary traces, 
$\tPi$ does not map onto functions with zero trace. 
Hence, we can extend~$\Pi^*$ only to~$L^1(\Omega)$-functions.
On the other hand, by Proposition~\ref{pro:localWeightsMod}~\ref{itm:dualBasisMod} $\Pi$ is a projection onto $\mathcal{L}^1_k(\tria)$. 
Using this additional property together with properties in Proposition~\ref{pro:localWeightsMod} (which are similar to the ones in Proposition~\ref{pro:localWeights}) in the proofs of Section~\ref{subsec:ProofThm1} leads to the statements for $\Pi$ in Theorems~\ref{thm:interpolOperator} and~\ref{thm:PropPi*}.

\begin{remark}[Zero trace on part of the boundary]
A modification of our design leads to a projection operator onto $\mathcal{L}^1_{k,\Gamma}(\tria) \coloneqq \lbrace v_h \in  \mathcal{L}^1_{k}(\tria) \colon v_h|_\Gamma = 0\rbrace$, where $\Gamma \subset \partial \Omega$ denotes a part of the boundary that is resolved by the triangulation.
The resulting operator has the approximation properties displayed in Theorem~\ref{thm:interpolOperator}, where we can even replace the space $W^{1,q}_0(\Omega)$ by $W^{1,q}_\Gamma(\Omega)\coloneqq \lbrace v\in W^{1,q}(\Omega)\colon v|_\Gamma = 0\rbrace$.
\end{remark}

\section{Alternative operators}\label{sec:AlternativeOp}
In this section we discuss alternative operators defined on $W^{-1,2}(\Omega)$  mapping to $\mathcal{L}^1_{k,0}(\tria)$ that share some of the properties of the projection $\Pizero$. 
We are interested in operators that are self-adjoint, projections, local, and have approximation properties.  
Note that the operator $\Pizero$ has all those properties except of being self-adjoint. 
However, it is not possible to add this property, without giving up one of the others, see Lemma~\ref{lem:noLocSelfAdjProj}. 
In the following sections we present selected operators that satisfy some but not all of the properties mentioned. 

\subsection{Self-adjoint and projection}\label{sec:SelfAdjProj}
We start by noting that there is only one self-adjoint projection onto $\mathcal{L}^1_{k,0}(\tria)$, namely the $L^2$-projection $\Pi_2$ determined by
\begin{align}\label{eq:DefL2proj}
  \langle \Pi_2 v, w_h \rangle_\Omega  = \langle v ,w_h\rangle_\Omega \qquad\text{for all } v \in L^2(\Omega)\text{ and } w_h \in \mathcal{L}_{k,0}^1(\tria).
\end{align}

\begin{lemma}[Uniqueness]\label{lem:noLocSelfAdjProj}
  Every self-adjoint projection $P \colon L^2(\Omega) \to \mathcal{L}^1_{k,0}(\tria)$ is the $L^2$-projection $\Pi_2$.
\end{lemma}
\begin{proof}
  Let $P \colon L^2(\Omega) \to \mathcal{L}^1_{k,0}(\tria)$ be a self-adjoint projection. 
  We have 
  \begin{align*}
    \skp{P v}{w_h}_\Omega = \skp{v}{P w_h}_\Omega = \skp{v}{w_h}_\Omega \qquad\text{for all } v \in L^2(\Omega), w_h \in  \mathcal{L}^1_{k,0}(\tria).
  \end{align*}
  This defines the $L^2$-projection to $\mathcal{L}^1_{k,0}(\tria)$ and hence we have that $P = \Pi_2$.
\end{proof}
Since $\mathcal{L}^1_{k,0}(\tria) \subset W^{1,p'}_0(\Omega)$ the $L^2$-projection $\Pi_2$ extends to $\xi \in W^{-1,p}(\Omega)$ for any $p\in(1,\infty)$. 
Due to its global definition the $L^2$-projection is not local. 
This means that the support of the projection of a local function is in general global. 
Nevertheless, this operator is useful in some situations. 
For uniform meshes, as well as for certain graded meshes, the $L^2$-projection is $W^{1,p}_0(\Omega)$-stable and thus by duality also stable in $W^{-1,p}(\Omega)$. This includes a wide class of adaptively generated meshes, for example meshes generated by the newest vertex bisection, see \cite{CT.1987,BPS.2002,Ca.2002,GHS.2016,BankYserentant14,DieningStornTscherpel21} for more details. 
Note that $W^{1,2}$-stability fails for general graded meshes, as shown by Bank and Yserentant \cite[Sec.~7]{BankYserentant14}.

\subsection{Self-adjoint and local}\label{sec:SelfAdjLoc}

In this subsection we focus on self-adjoint and local operators. 
Both properties are very useful for the numerical analysis of various problems. 
The lowest order case $k = 1$ is considered for example in \cite{Carstensen99,CarstensenVerfuerth99,BPV.2000,MalqvistPeterseim14} featuring
the operator~$\interC\colon L^1(\Omega) \to \mathcal{L}^1_{1,0}(\tria)$ with
\begin{align}\label{eq:OperatorCC}
  \interC v \coloneqq \sum_{i\in\nodesint} \frac{\langle v ,b_i \rangle_\Omega}{\langle 1,b_i \rangle_\Omega}\,b_i\qquad\text{for all }v \in L^{1}(\Omega).
\end{align}
The operator is $W^{1,2}_0(\Omega)$-stable, and therefore by duality also $W^{-1,2}(\Omega)$-stable. It has the approximation property 
\begin{align*}
  \norm{w - \interC w  }_{L^2(T)} \lesssim h \norm{\nabla w} _{L^2(\omega_T)}\qquad\text{for all }w \in W^{1,2}_0(\Omega)\text{ and }T \in \tria.
\end{align*}
Since $\interC$ is not a projection, it does not preserve linear functions and consequently one has in general a reduced maximal order of convergence in the sense that
\begin{align*}
  \norm{w - \interC w  }_{L^2(T)} 
  \not\lesssim h^2 \norm{\nabla^2 w}_{L^2(\omega_T)}
  \quad\text{for } w\in W^{1,2}_0(\Omega)\cap W^{2,2}(\Omega)\text{ and }T \in \tria.
\end{align*}
The local self-adjoint operator $\interC$ was generalized to polynomial degrees $k\in \mathbb{N}$ in \cite[Eq.~3.4]{DieningStornTscherpel21}, which we discuss in the following. 
We restrict the presentation to the case of zero boundary traces, since in this case the self-adjoint operator extends to a mapping on~$W^{-1,2}(\Omega) = (W^{1,2}_0(\Omega))^*$. 
The design of~$\interC$ for~$k\in \setN$ is based on local weighted projections. 
Let $\vertices$ denote the set of vertices in $\tria$ and let $\phi_j \in \mathcal{L}^1_1(\tria)$ denote the nodal basis function associated to the vertex $j \in \vertices$ and $\omega_j = \support(\varphi_j) \subset \overline{\Omega}$.
For $j\in\mathcal{V}$ and $v_{k-1} \in \mathcal{L}^1_{k-1}(\tria(\omega_j))$ we denote by $\phi_j v_{k-1}\in \mathcal{L}^1_k(\tria)$ the function where we extend $v_{k-1}$ outside of $\omega_j$ by zero.
We set for all $j\in \mathcal{V}$ the space of local Lagrange functions with zero boundary traces on the boundary $\partial \Omega$ by 
\begin{align}\label{def:bdLag}
  \begin{aligned}
    \mathcal{L}^1_{k-1,\partial \Omega}(\tria(\omega_j)) &\coloneqq \set{ v_{k-1} \in \mathcal{L}^1_{k-1}(\tria(\omega_j))\colon \phi_j v_{k-1} \in \mathcal{L}^1_{k,0}(\tria)}.
  \end{aligned}
\end{align}
Let $\interC_j \colon L^2(\omega_j) \to \mathcal{L}^1_{k-1,\partial \Omega}(\tria(\omega_j))$ with 
$j\in \mathcal{V}$ denote the orthogonal projection with respect to the weighted inner product $\langle \bigcdot,\bigcdot\rangle_{\varphi_j} \coloneqq \langle \phi_j \bigcdot,\bigcdot\rangle_{\omega_j}$. 
By the definition in \eqref{def:bdLag} we have $\phi_j \mathcal{L}^1_{k-1,\partial \Omega}(\tria(\omega_j)) \subset \mathcal{L}^1_{k,0}(\tria)$. Hence the mapping $\interC_j$ extends to an operator on $W^{-1,p}(\Omega)$ mapping to $\mathcal{L}^1_{k-1,\partial \Omega}(\tria(\omega_j))$ defined by
\begin{align}\label{eq:DefCz}
  \langle \interC_j \xi,v_{k-1}\rangle_{\varphi_j} = \langle \xi,v_{k-1}\rangle_{\varphi_j}\quad\text{for all }\xi \in W^{-1,p}(\Omega),v_{k-1}\in \mathcal{L}^1_{k-1,\partial \Omega}(\tria(\omega_j)).
\end{align}
The combination of the local operators $\interC_j$ leads to the operator 
\begin{align}\label{eq:DefC}
  \interC\colon W^{-1,p}(\Omega) \to \mathcal{L}^1_{k,0}(\tria) \qquad \text{with}\qquad\interC \coloneqq \sum_{j\in \vertices} \phi_j \interC_j.
\end{align}
Notice that for $k = 1$ this operator coincides with the operator defined in~\eqref{eq:OperatorCC}. 
\begin{lemma}[Properties of $\interC$]\label{lem:propC}
  The operator $\interC$ satisfies the following properties. 
  \begin{enumerate}
  \item \label{itm:Prop1} 
    It is linear and self-adjoint with respect to~$\langle \bigcdot,\bigcdot\rangle_\Omega$.
  \item \label{itm:Prop2}
    It is $L^2$-elliptic on $\mathcal{L}^1_{k,0}(\tria)$ in the sense that for all $v_h \in \mathcal{L}^1_{k,0}(\tria)$
    \begin{align*}
      \frac{k}{2k + d}\, \norm{v_h}^2_{L^2(\Omega)} \leq  \skp{ \interC v_h }{v_h}_{\Omega} \leq \norm{ v_h }^2_{L^2(\Omega)},
      \\
      \frac{k}{2k + d}\, \norm{ v_h}_{L^2(\Omega)} \leq  \norm{ \interC v_h }_{L^2(\Omega)} \leq \norm{ v_h }_{L^2(\Omega)}.
    \end{align*}
  \item \label{itm:Prop2b} One has $\norm{\interC v}_{L^2(\Omega)} \leq \norm{v}_{L^2(\Omega)}$ for all $v \in L^2(\Omega)$.
  \item \label{itm:Prop3}
    It is the identity on $\mathcal{L}^1_{k-1,0}(\tria)$. 
  \item \label{itm:Prop4}
    The difference $\identity -\interC$ is orthogonal on $\mathcal{L}^1_{k-1,0}(\tria)$, that is,
    \begin{align*}
      \langle \xi - \interC \xi, v_{k-1} \rangle_\Omega  = 0 \quad\text{for all }\xi\in W^{-1,p}(\Omega) \text{ and }v_{k-1}\in \mathcal{L}^1_{k-1,0}(\tria).
    \end{align*} 
  \item \label{itm:Prop5}
    It is local in the sense that for all $T\in \tria$ and $v\in L^p(\Omega)$ one has
    \begin{align*}
      \textup{supp}\big(\interC (v\indicator_T)\big) \subset \omega_T.
    \end{align*}
  \item \label{itm:Prop6} It satisfies $\interC \Pi_2 = \interC = \Pi_2 \interC$, where $\Pi_2$ is $L^2$-projection to $\mathcal{L}^1_{k,0}(\tria)$  as in~\eqref{eq:DefL2proj}. 
  \item \label{itm:Prop7} The operator~$\interC$ preserves constants on interior simplices, that is,
    \begin{align*}
      (\interC 1)|_T = 1 \qquad \text{ for all }\, T\in \tria \text{ with }T\cap \partial \Omega = \emptyset.
    \end{align*}
  \end{enumerate} 
\end{lemma}
\begin{proof}
  The proof of~\ref{itm:Prop1}--\ref{itm:Prop6} is presented in \cite[Sec.~3]{DieningStornTscherpel21}. 
For the proof of~\ref{itm:Prop7} note that $\mathcal{L}^1_{k-1,\partial\Omega}(\tria(\omega_j)) = \mathcal{L}^1_{k-1}(\tria(\omega_j))$ for any $j \in \verticesint$. 
 Then the claim follows from $\interC_j 1 = 1$ for all $j \in \verticesint$ and $(\sum_{j \in \verticesint} \phi_j)|_T =1$ for $T\in \tria$ such that $T \cap \partial \Omega = \emptyset$.
\end{proof}
\begin{theorem}[Interpolation error]\label{thm:interC}
 Let $k \in \mathbb{N}$ be arbitrary.   
  The operator~$\interC$ has the following properties, for any $j\in \mathcal{V}$ and $T\in \tria$: 
  
  \noindent%
  \textbf{\,Localization of norms.} For $p \in (1,\infty)$ 
      there holds
  \begin{align}
       \norm{\xi - \interC \xi }_{W^{-1,p}(\Omega)} & \eqsim \Big(\sum_{j \in \mathcal{V}} \norm{\xi- \interC \xi}_{W^{-1,p}(\omega_j)}^p \Big)^{1/p}\quad\text{for all }\xi \in W^{-1,p}(\Omega).
  \end{align}
  
  \noindent%
  \textbf{\,Approximabilty.} For $p \in (1,\infty)$, $q \in [1,\infty]$ and $0 \leq m \leq s \leq k$ one has
  \begin{align}\label{eq:ApxIlocC}
    \begin{alignedat}{2}
      \norm{v - \interC v }_{W^{-1,p}(\omega_j)} & \lesssim   h_j \norm{v}_{L^{p}(\omega_j^2)}  &&\text{for all } v \in L^{p}(\Omega),\\
      \norm{w - \interC w }_{W^{-1,p}(\omega_j)}&\lesssim  h_j^{s+1} \norm{\nabla^s w}_{L^p(\omega^2_j)} 
      \qquad 
      &&\text{for all } w\in  W^{1,p}_0(\Omega) \cap W^{s,p}(\Omega),
      \\
      \norm{\nabla^m (w - \interC w)}_{L^q(T)} &\lesssim h_T^{s-m} \norm{ \nabla^s w }_{L^q(\omega_T)}
     \qquad &&\text{for all } w \in W^{1,q}_0(\Omega)\cap W^{s,q}(\Omega).
    \end{alignedat}
  \end{align}
  
  \noindent%
  \textbf{\,Local Stability.}  For $p \in (1,\infty)$, $q \in [1,\infty]$ and $0 \leq s \leq k$ there holds
  \begin{align}\label{eq:LocStabC}
    \begin{alignedat}{2}
      \norm{\interC \xi}_{W^{-1,p}(\omega_j)} &\lesssim \norm{\xi}_{W^{-1,p}(\omega_j^2)} &&\text{for all }\xi \in W^{-1,p}(\Omega),
      \\
      \lVert \interC v \rVert_{L^q(T)}& \lesssim \lVert v \rVert_{L^q(\omega_T)} &&\text{for all }v\in L^q(\Omega),
      \\
      \lVert \nabla^s\interC w \rVert_{L^q(T)}& \lesssim \norm{ \nabla^s w }_{L^q(\omega_T)} \qquad &&\text{for all }w\in W^{1,q}_0(\Omega)\cap W^{s,q}(\Omega).
    \end{alignedat}
  \end{align}

  \noindent%
  \textbf{\,Global Stability.}   For $p \in (1,\infty)$, $q \in [1,\infty]$ and $0 \leq s \leq k$ one has
  \begin{align}\label{eq:GlobStabC}
    \begin{alignedat}{2}
      \norm{\interC \xi}_{W^{-1,p}(\Omega)} &\lesssim \norm{\xi}_{W^{-1,p}(\Omega)} &\qquad&\text{for all }\xi \in W^{-1,p}(\Omega),
      \\
      \norm{ \interC v }_{L^q(\Omega)}& \lesssim \lVert v \rVert_{L^q(\Omega)} &&\text{for all }v\in L^q(\Omega),
      \\
      \lVert \nabla^s\interC w \rVert_{L^q(\Omega)}& \lesssim \lVert \nabla^s w \rVert_{L^q(\Omega)}&&\text{for all }w\in W^{1,q}_0(\Omega)\cap W^{s,q}(\Omega).
    \end{alignedat}
  \end{align}
  The hidden constants are independent of $p$ and~$q$.
\end{theorem}
\begin{proof}
We start with the verification of the $L^q$-stability in \eqref{eq:GlobStabC}. Let $v\in L^\infty(\Omega)$.
  By the definition in~\eqref{eq:DefC} we obtain
  \begin{align*}
    \norm{ \interC v }_{L^\infty(\Omega)} \leq \Big\lVert\sum_{j\in \vertices} \varphi_j \norm{ \interC_j v }_{L^\infty(\omega_j)}\Big\rVert_{L^\infty(\Omega)} \leq \max_{j\in \vertices}\,\norm{ \interC_j v }_{L^\infty(\omega_j)}.
  \end{align*}
  Let $j\in \vertices$. 
  Then by an inverse estimate and the shape regularity of $\tria$ we have that
  \begin{align*}
    \norm{ \interC_j v }_{L^\infty(\omega_j)} = \max_{T\in \tria(\omega_j)} \norm{ \interC_j v }_{L^\infty(T)} 
    \lesssim \abs{\omega_j}^{-1/2} \Big( \sum_{T\in \tria(\omega_j)} \lVert\varphi_j^{1/2} \interC_j v \rVert_{L^2(T)}^2 \Big)^{1/2}.
  \end{align*}
  Due to~\eqref{eq:DefCz} the sum can be bounded as 
  \begin{align*}
    \sum_{T\in \tria(\omega_j)} \lVert \varphi_j^{1/2} \interC_j v \rVert_{L^2(T)}^2 
    = \langle \interC_j v , v\rangle_{\varphi_j}
     \leq \norm{ \interC_j v }_{L^\infty(\Omega)} \norm{ v }_{L^\infty(\Omega)} \abs{\omega_j}. 
  \end{align*}
 Combining the previous three estimates shows 
  \begin{align*}
    \norm{ \interC v }_{L^\infty(\Omega)} \lesssim \norm{ v }_{L^\infty(\Omega)}.
  \end{align*}
  For $q\in [1,\infty]$ we conclude global $L^q$-stability  by interpolation with the $L^2$-stability, see Lemma~\ref{lem:propC}\,\ref{itm:Prop2b}, and duality. 
By locality of~$\interC$ the local $L^q$-estimates in \eqref{eq:LocStabC} follow.

The remaining statements of the theorem follow with similar arguments as in the proof of Theorem~\ref{thm:interpolOperator} for~$\Pizero$ in Section~\ref{subsec:ProofThm1} using the local $L^p$-stability and Lemma~\ref{lem:propC}. 
The reduced maximal order of approximability $s\leq k$ in~\eqref{eq:ApxIlocC} is due the fact that~$\interC$ only preserves functions in $\mathcal{L}^1_{k-1,0}(\tria)$ but not functions in $\mathcal{L}^1_{k,0}(\tria)$.
\end{proof}

\subsection{Locality and projection}\label{sec:LocalProj}
Our operators $\Pizero$ and $\Pi$ as well as further Scott--Zhang type operators like the ones in \cite{ScottZhang90,TantardiniVeeser16} are local projections that are not self-adjoint. 
While the latter operators use discontinuous polynomial weights, we use continuous weight functions to ensure that the operator is well-defined on functionals in $W^{-1,2}(\Omega)$. 
Yet, the polynomial degree $3k$ of our weights is much higher than the polynomial degree $k$ of the discontinuous weights. 
A natural question to ask is whether it is possible to reduce the polynomial order of the continuous weight functions and obtain a local projection with properties as in Theorem~\ref{thm:interpolOperator}. 
In the following we argue that we need $\mathcal{L}^1_{k+1,0}(\tria)$ weights to ensure locality of the projection. 
Furthermore, we illustrate that projections with similar locality and $W^{-1,2}$-approximation properties as $\Pizero$ require even larger degrees than $k+1$. In particular, we show that the degree $3k$ is optimal for $k=1=d$. 

Let $P \colon L^2(\Omega) \to \mathcal{L}^1_{k,0}(\tria)$ be a linear and bounded operator.  
As a consequence of the Riesz representation theorem, there exist unique weights $\rho_i \in L^2(\Omega)$ with $i \in \nodesint$ such that $P$ has a representation in terms of the Bernstein basis
\begin{align}\label{eq:P-Bernstein}
  Pv = \sum_{i \in \nodesint} \skp{v}{\rho_i}_\Omega\, b_i.
\end{align}
The adjoint operator $P^* \colon L^2(\Omega) \to L^2(\Omega)$ reads  
\begin{align}\label{eq:Padj-Bernstein}
  P^* w = \sum_{i \in \nodesint} \skp{w}{b_i}_\Omega\, \rho_i. 
\end{align}
\begin{lemma}[Projection]\label{lem:proj}
  Let $P \colon L^2(\Omega) \to \mathcal{L}^1_{k,0}(\tria)$ be a linear bounded operator with the representation in~\eqref{eq:P-Bernstein}.
  If $P$ is a projection, then one has that 
  \begin{align*}
    \skp{\rho_i}{b_j}_\Omega = \delta_{i,j} \quad \text{ for all } i,j \in \nodesint.
  \end{align*}
  Furthermore, trivially $P^*$ is a projection onto the span of $\{\rho_i \colon i \in \nodesint\}$.   
\end{lemma}
\begin{proof}
  Since $P$ is a projection, we have for all $j\in \nodesint$ that
  \begin{align*}
    Pb_j = \sum_{i \in \nodesint} \skp{b_j}{\rho_i}_\Omega\, b_i = b_j.  
  \end{align*}
  Because $(b_i)_{i \in \nodesint}$ is a basis, this implies that $\skp{b_j}{\rho_i}_\Omega = \delta_{i,j}$ for any $i,j \in \nodesint$. 
\end{proof}

The following lemma is a stronger version of Lemma~\ref{lem:noLocSelfAdjProj} showing that any projection with same order continuous weights is the $L^2$-projection. 
Since the $L^2$-projection is not local this implies that a projection to $\mathcal{L}^1_{k,0}(\tria)$ cannot be local when it has weight functions of the same polynomial degree $k$.

\begin{lemma}[Same order weights]\label{lem:L2-weights}
Let $P\colon L^2(\Omega)\!\to\!\mathcal{L}^1_{k,0}(\tria)$ be a projection onto $\mathcal{L}^1_{k,0}(\tria)$ with weights $\rho_i \in \mathcal{L}^1_{k,0}(\tria)$ in~\eqref{eq:P-Bernstein}. 
Then $P$ is the $L^2$-projection $\Pi_2$  in~\eqref{eq:DefL2proj}.  
\end{lemma}
\begin{proof} 
Let $P\colon L^2(\Omega)\!\to\!\mathcal{L}^1_{k,0}(\tria)$ be a projection with representation as in~\eqref{eq:P-Bernstein} and weights $\rho_i \in \mathcal{L}^1_{k,0}(\tria)$. 
Since the functions $(b_i)_{i \in \nodesint}$ form a basis of $\mathcal{L}^1_{k,0}(\tria)$, the biorthogonality in Lemma~\ref{lem:proj} implies that also the functions $(\rho_i)_{i \in \nodesint}$ form a basis of $\mathcal{L}^1_{k,0}(\tria)$.  
 Thus, the adjoint operator  $P^*$ as represented in~\eqref{eq:Padj-Bernstein} is a projection onto $\mathcal{L}^1_{k,0}(\tria)$.  
Hence, for any $v \in L^2(\Omega)$ and any $w_h \in \mathcal{L}^1_{k,0}(\tria)$ we find that
\begin{align*}
\skp{v - Pv }{w_h}_{\Omega} = \skp{v}{w_h - P^* w_h}_{\Omega}  = 0. 
\end{align*}
This is exactly the definition of the $L^2$-projection and thus finishes the proof. 
\end{proof}

It is possible to construct a local projection with weight functions in $\mathcal{L}^1_{k+1,0}(\tria)$ as the following remark shows. 

\begin{remark}[$k+1$ weights]\label{rmk:weights}
  For given polynomial degree $k\geq 1$ let $\Pizero$ be our $W^{-1,2}$-stable projection defined in~\eqref{eq:def-Pi} and $\interC_{k+1}$ the interpolation operator defined in~\eqref{eq:DefC} mapping to $\mathcal{L}^1_{k+1,0}(\tria)$. 
  Then $\Pizero$ and $\interC_{k+1}$ are both $W^{-1,2}$-stable and preserve $\mathcal{L}^1_{k,0}(\tria)$-functions.  Thus, the composition $\Pizero \circ \interC_{k+1}$ is a projection onto~$\mathcal{L}^1_{k,0}(\tria)$  which is $W^{-1,2}$-stable. The weights of~$\interC_{k+1}$ are $\mathcal{L}^1_{k+1,0}(\tria)$-functions, so the weights of~$\Pizero \circ \interC_{k+1}$ are also $\mathcal{L}^1_{k+1,0}(\tria)$-functions with increased support. 
  For the operator $\Pizero \circ \interC_{k+1}$ we obtain similar properties as in Theorem~\ref{thm:interpolOperator} but with~$\omega_i^2$ replaced by~$\omega_i^3$.
  \end{remark}

\begin{remark}[Alternative weights]\label{rmk:op-Stevenson}
Further suitable weight functions with slightly enlarged support have been designed in \cite{SV.2020}. 
The authors' approach relies on an element bubble correction. 
For the weight functions one can use for example functions in $\mathcal{L}^1_{k+d+1}(\tria)$ or weight functions in $\mathcal{L}^1_{k+1}(\tria')$, where $\tria'$ is a full barycentric refinement of $\tria$. 
The resulting projection operator $\widetilde{\Pi}$ is local in the sense that $\support(\widetilde{\Pi} (v \indicator_T)) \subset \omega_T^2$ for any $T \in \tria$. 
Hence it has the same locality property as the operator in Remark~\ref{rmk:weights} but uses higher polynomial degrees for the weight functions. 
\end{remark}

The operator in the previous remark is not as local as the projection operator $\Pizero$ with weights of degree $3k$. 
The following remark shows that for the special case of $k = d = 1$ one cannot take a lower polynomial degree for an operator as local as $\Pizero$ such that its adjoint preserving constants on interior simplices. 
Recall that the latter is a key in the proof of Theorem~\ref{thm:interpolOperator}. 
In particular, it does not suffice to choose weights of polynomial degree $k+1$. 

\begin{remark}[Increased order, $d = k = 1$]
  We consider a regular triangulation that contains the interior simplex~$T=[0,1]$. 
  Assume that there are weight functions $\rho_0,\rho_1\in \mathcal{L}^1_{2,0}(\tria)$ which satisfy that $\rho_i(j) = 0$, if $i \neq j$, to ensure the locality. 
  This requirement and symmetry show that on $T$ the weights $\rho_0$ and~$\rho_1$ are
  of the form
  \begin{align*}
    \rho_{0} = \lambda_0(s \lambda_0 + t \lambda_1) \quad  \text{ and } \quad 
    \rho_{1} = \lambda_1(s \lambda_1 + t \lambda_0) \qquad \text{ for some }s, t\in \mathbb{R}. 
  \end{align*}
  The property $\Pi^*1 = 1$ requires
  $\rho_0+\rho_1 = c_k^{-1}=2$ on~$T$, which determines $s=1$ and $t=2$.
 It follows that $\skp{\rho_0}{b_1}_\Omega = \skp{\rho_0}{\lambda_1}_T  >0$. This contradicts the biorthogonality in Lemma~\ref{lem:local-basis}\,\ref{itm:local-duality}.   
\end{remark}

\section{Applications}
\label{sec:applications}

This section contains applications of the projections~$\Pi$ and~$\Pizero$ to illustrate their beneficial properties. 
We investigate a semi- and a full discretization of the heat equation as well as a least squares finite element method with rough given data. 

\subsection{Interpolation in semi-discrete time marching schemes}
\label{sec:InterpolSemiDiscr}
We consider the heat equation on a bounded time space cylinder $Q = \mathcal{J} \times \Omega$ with bounded time interval $\mathcal{J} = (0,T)$ and a bounded Lipschitz domain $\Omega \subset \mathbb{R}^d$. 
For given right-hand side $f\colon Q \to \mathbb{R}$ and initial data $u_0\colon \Omega \to \mathbb{R}$ we aim to find a function $u$ satisfying
\begin{alignat}{2}
  \partialt u -\Delta_\bfx u& = f\ \quad &&\text{ in }Q, \notag\\
  u(0,\bigcdot) &= u_0\ \;&&\text{ in }\Omega,\label{eq:ExampleHeat} \\
  u &= 0\ \;&&\text{ on } \mathcal{J} \times \partial \Omega. \notag
\end{alignat}
With the standard notation for Bochner spaces we assume that $f\in L^2(\mathcal{J};W^{-1,2}(\Omega))$ and $u_0 \in L^2(\Omega)$. 
The unique weak solution can be found in the space
\begin{align}
  \label{eq:def-V-parabolic}
  V \coloneqq L^2(\mathcal{J};W^{1,2}_0(\Omega)) \cap W^{1,2}(\mathcal{J};W^{-1,2}(\Omega)).
\end{align}

In the following we consider a semi-discretization in time. More specifically, for a regular partition $\triax$ of $\Omega$ into closed simplices our semi-discrete ansatz space reads $V_{h_\bfx} \coloneqq W^{1,2}(\mathcal{J};\mathcal{L}^1_{k,0}(\triax))\subset V$. 
Let $\Pi_2\colon L^2(\Omega) \to \mathcal{L}^1_{k,0}(\triax)$ denote the $L^2$-orthogonal projection mapping to $\mathcal{L}^1_{k,0}(\triax)$. 
Then the discrete solution $u_h \in V_{h_\bfx}$ satisfies $u_h(0,\bigcdot) = \Pi_2 u_0$ and for all $w_h \in \mathcal{L}^1_{k,0}(\triax)$ and for a.e.~$s\in \mathcal{J}$
\begin{align}\label{eq:ExampleHeatSemiDis}
  \langle \partialt u_h(s),w_h\rangle_\Omega + \langle \nabla_\bfx u_h(s), \nablax w_h \rangle_\Omega = \langle f(s),w_h\rangle_\Omega.
\end{align}
While classical error analysis for such problems treats rate-optimality, see~\cite{Thomee06}, more recent results include also quasi-optimality of the semi-discrete approximation. 
In \cite[Thm.~3.4]{CH.2002} and \cite[Thm.~3.10]{TantardiniVeeser16} the authors prove that under the (necessary) assumption that the $L^2$-projection $\Pi_2$ is $W_0^{1,2}(\Omega)$-stable it holds that
\begin{align}\label{eq:BestApxTanVee}
  \begin{aligned}
    &\norm{ \nablax (u - u_h) }_{L^2(\mathcal{J};L^2(\Omega))}^2 + \norm{ \partialt (u - u_h) }^2_{L^2(\mathcal{J};W^{-1,2}(\Omega))} \\
    &\qquad\lesssim \min_{v_h \in V_{h_\bfx}} \left( \norm{ \nablax (u - v_h) }_{L^2(\mathcal{J};L^2(\Omega))}^2 + \norm{ \partialt (u - v_h) }_{L^2(\mathcal{J};W^{-1,2}(\Omega))}^2 \right).
  \end{aligned}
\end{align}
Refer to Section~\ref{sec:SelfAdjProj} for details on~$\Pi_2$ and its $W^{1,2}_0(\Omega)$-stability.
The properties of our projection $\Pizero$, as introduced in Section~\ref{sec:higher-order}, allow us to reprove the following error estimate in \cite[Eq.~5.10]{TantardiniVeeser16}. 
Let $\mathcal{V}_x$ denote the set of vertices in~$\triax$. 
\begin{theorem}[A~priori error estimate]\label{thm:APrioriTV}
  If $\Pi_2$ is $W^{1,2}_0(\Omega)$-stable, then for the solution $u$ to~\eqref{eq:ExampleHeat} and the semi-discrete solution $u_h$ to~\eqref{eq:ExampleHeatSemiDis} and all $r = 1,\dots, k+1$ and $s = 0,\dots, k+1$ one has that
  \begin{align*}
    & \norm{ \nablax (u - u_h) }^2_{L^2(\mathcal{J};L^2(\Omega))} + \norm{ \partialt (u - u_h) }_{L^2(\mathcal{J};W^{-1,2}(\Omega))}^2 \\
    & \qquad\lesssim \sum_{j\in \mathcal{V}_x}  \left( h_j^{2(r-1)} \norm{\nablax^{r} u}_{L^2(\mathcal{J};L^2(\omega_j^2))}^2 + h_j^{2(s+1)} \norm{ \partialt \nablax^{s}u }^2_{L^2(\mathcal{J};L^2(\omega_j^2))} \right). 
  \end{align*}
\end{theorem}
\begin{proof}
  For any $v \in V$ we denote by $\Pizero v \in L^2(\mathcal{J};\mathcal{L}^1_{k_x}(\triax))$ the pointwise in time application of the projection operator $\Pizero$, introduced in Section~\ref{sec:higher-order}.  
  For smooth functions $v \in V$ we obtain the commutation property $\partial_t \Pizero v = \Pizero \partial_t v $. 
  Due to density and the $W^{-1,2}$-stability in Theorem~\ref{thm:interpolOperator} this extends to arbitrary functions  $v \in V$. 
With the best approximation property~\eqref{eq:BestApxTanVee} with $v_h \coloneqq \Pi u$, the commutation property, and the interpolation error estimates in Theorem~\ref{thm:interpolOperator} the claim follows. 
\end{proof}

\begin{remark}[Fractional orders]
  Using fractional estimates for our interpolation operator (see Remark~\ref{rem:FracOrder}) allows us to extend the result of Theorem~\ref{thm:APrioriTV} to fractional orders, similarly as in \cite[Prop.~7.27]{Tantardini12} for polynomial degree $k = 1$.  
\end{remark}

\subsection{Interpolation on tensor meshes in space-time domains}\label{sec:InterPolTensor}

So-called space-time finite element methods constitute an alternative to time-marching schemes for parabolic problems.  
Space-time methods treat the time as an additional spatial dimension and apply techniques known from time-independent problems, see for example \cite{Steinbach15,LangerMooreNeumueller16,StevensonWesterdiep20,DieningStorn21}.     
The resulting schemes are quasi-optimal and hence we can bound the best-approximation error from above by an interpolation error. 
This results in optimal rates of convergence for the numerical schemes and motivates adaptive mesh refinement. 
However, existing interpolation operators as in \cite[Sec.~4.1]{FuehrerKarkulik19} or \cite[Cor.~3.4]{Steinbach15} require smooth solutions and are not stable with respect to the norm in the space $V$ given by
\begin{align*}
  \norm{ \bigcdot }_{V} \coloneqq 
  \norm{ \partialt \bigcdot }_{L^2(\mathcal{J};W^{-1,2}(\Omega))} + \norm{ \nablax \bigcdot }_{L^2(\mathcal{J};L^2(\Omega))}.
\end{align*}  
We remedy this difficulty for discrete tensor product subspaces. Let $\triat$ denote a partition of the time interval $\mathcal{J}$ into closed intervals and let $\triax$ denote a regular partition of the domain $\Omega$ into closed simplices.
The tensor product mesh $\mathcal{Q} \coloneqq \triat \otimes \triax$ consists of closed time-space cells $K = K_\bft \times K_\bfx$ with $K_\bft\in \triat$ and $K_\bfx \in \triax$. 
For polynomial degrees $k_\bft ,k_\bfx \in \mathbb{N}$ the discrete tensor product subspace is given by
\begin{align}\label{eq:TensorFEspace}
  &V_h \coloneqq \mathcal{L}^1_{k_\bft}(\triat) \otimes \mathcal{L}^1_{k_\bfx,0}(\triax) \subset V.
\end{align}
In other words, each element $v_h \in V_h$ is a continuous piecewise polynomial function. 
More specifically, for each time-space cell $K = K_\bft \times K_\bfx \in \mathcal{Q}$ there are polynomials $v_\bft \in \mathcal{P}_{k_\bft}(K_\bft)$ and $v_\bfx \in \mathcal{P}_{k_\bfx}(K_\bfx)$ with $v_h|_K = v_\bft v_\bfx$. 
Similarly, we define the time-discrete space by
\begin{align*}
  V_{h_\bft} & \coloneqq \mathcal{L}^1_{k_\bft}(\triat) \otimes L^2(\Omega) \subset V. 
\end{align*}
Let $\Pi_\bft\colon L^2(\mathcal{J}) \to \mathcal{L}^1_{k_\bft}(\triat)$ denote the 1D Scott--Zhang type interpolation operator defined in~\eqref{eq:def-Pi} without zero boundary traces. 
More precisely, let $b_{\bft,i}\in \mathcal{L}^1_{k_\bft}(\triat)$ denote the Bernstein basis functions associated to the Lagrange node $i \in \nodest$ of $\triat$ and let $\psi_{\bft,i} \in \mathcal{L}^1_{3k_t}(\triat)$ denote the biorthogonal basis functions designed in Section~\ref{sec:higher-order}. 
Then the operator reads 
\begin{align*}
  \Pi_\bft v_\bft \coloneqq \sum_{i\in \nodest} \langle v_\bft, \psi_{\bft,i} \rangle_\mathcal{J}\, b_{\bft,i}\qquad\text{for all }v_\bft \in L^2(\mathcal{J}).
\end{align*}
We extend this operator to a mapping $\Pi_\bft\colon L^2(\mathcal{J};L^2(\Omega)) \to V_{h_\bft}$ 
defined by 
    %
\begin{align}\label{eq:DefPIt}
  \Pi_\bft v = \sum_{i\in \nodest} \langle v, \psi_{\bft,i}\rangle_\mathcal{J} \, b_{\bft,i}\qquad\text{for all } v \in L^2(\mathcal{J};L^2(\Omega)).
\end{align}
Let $\omega_{K_\bft} = \bigcup \lbrace K_\bft' \in \triat\colon K_\bft \cap K_\bft' \neq \emptyset \rbrace$ denote the element patch with respect to the partition in time for all $\Kt\in \triat$. 
We assume that $\triat$ is shape regular in the sense that we have for all $T\in \triat$ the equivalence $h_\Kt \coloneqq \abs{\Kt} \eqsim \abs{\omega_{K_\bft}}$.
\begin{theorem}[Interpolation in time]\label{thm:InterpolInTime}
  Let $X$ denote the space $W^{-1,2}(\omega)$ with $\omega \subset \Omega$ or $L^2(\Kx)$ with $\Kx \in \triax$.      
  The projection $\It\colon L^2(\mathcal{J};L^2(\Omega)) \to V_{h_\bft}$ defined in~\eqref{eq:DefPIt} satisfies for all $v\in V$, all $m,s\in \mathbb{N}_0$ with $0\leq m\leq s \leq k_\bft + 1$, and all $\Kt\in \triat$ 
  \begin{align}\label{eq:LocalLpStabInTime}
    \norm{ \partialt^m (v - \Pi_\bft v) }_{L^2(K_\bft;X)} \lesssim h_\Kt^{s-m} \norm{\partialt^s v}_{L^2(\omega_{\Kt};X)}.
  \end{align}
\end{theorem}
\begin{proof}
  Let $X$ be defined as in the theorem. 
  Let $v\in W^{s,2}(I;X)$ for some $s\in \mathbb{N}_0$, let $m \in \mathbb{N}_0$ with $m \leq s$, and let $\tau \in \Kt \in \triat$.
  The definition of $\Pi_\bft$ in~\eqref{eq:DefPIt} yields
  \begin{align*}
    \partialt^m \Pi_\bft v(\tau) 
    = \sum_{i\in \nodest} \langle v, \psi_{\bft,i} \rangle_{\mathcal{J}} \, \partialt^m b_{\bft,i}(\tau) 
    = \sum_{i\in \nodest \cap K_t} \langle v, \psi_{\bft,i} \rangle_{\mathcal{J}} \, \partialt^m b_{\bft,i}(\tau).
  \end{align*}
  Hence, Jensen's inequality, $\support(\psi_{t,i})\subset \omega_{t,i} \coloneqq \support (b_{t,i})$, and the fact that by Proposition~\ref{pro:localWeights} we have $\norm{\psi_{t,i}}_{L^{2}(\omega_{t,i})} \lesssim \abs{K_t}^{-1/2}$, yield
  \begin{align*}
    \norm{ \partialt^m \Pi_\bft v(\tau) }_X 
    &\leq \sum_{i\in\nodest \cap K_t} \int_{\mathcal{J}} \norm{v(\sigma)}_X \abs{\psi_{\bft,i}(\sigma)} \dsig\, \abs{\partialt^m b_{\bft,i}(\tau)}\\
    & \lesssim \abs{\Kt}^{-m} \sum_{i\in\nodest \cap K_t} \int_{\omega_{t,i}} \norm{v(\sigma)}_X \abs{\psi_{\bft,i}(s)} \ds\\
    & \leq \abs{\Kt}^{-m} \sum_{i \in \nodest \cap K_t} \norm{v}_{L^2(\omega_{t,i};X)} \norm{ \psi_{\bft,i}}_{L^{2}(\omega_{t,i})} \\
    &\lesssim \abs{\Kt}^{-m} \abs{\Kt}^{-1/2} \norm{v}_{L^2(\omega_\Kt;X)}.
  \end{align*}
  Integrating over $\Kt$ results in 
  \begin{align}\label{eq:inverseLikeEst}
    \norm{\partialt^m \Pi_\bft v}_{L^2(\Kt; X)} \lesssim \abs{\Kt}^{-m}  \norm{v}_{L^2(\omega_\Kt;X)}.
  \end{align}
  We use the averaged Taylor polynomial $T^{k_t} v \in \mathcal{P}_{k_t} (K_t) \otimes X$ of degree~$k_t$ in Bochner spaces. 
  The definition and approximation properties of $T^{k_t}$ are presented in Lemma~\ref{lem:PolyApxInTime} below. 
  From this and the inequality in~\eqref{eq:inverseLikeEst} it follows that  
    \begin{align*}
    &\norm{\partialt^m (v - \Pi_\bft v)}_{L^2(\Kt;X)} \\
    &\qquad \leq 
    \norm{ \partialt^m (v - T^{k_t} v) }_{L^2(\Kt;X)} + \norm{ \partialt^m \Pi_\bft (v - T^{k_t} v)}_{L^2(\Kt;X)}\\
    &\qquad \lesssim \abs{\Kt}^{s-m} \norm{\partialt^s v}_{L^2(\omega_\Kt;X)}\qquad\qquad\text{for all }s = m,\ldots, k_t+1.\qedhere
  \end{align*}
\end{proof}
In addition to an interpolation operator in time we now want to apply an interpolation operator $\Ix\colon  V \to V_{h_\bfx}\coloneqq L^2(\mathcal{J}) \otimes \mathcal{L}^1_{k_\bfx,0}(\triax)$ in space. 
This operator results from applying the projection $\Pi\colon W^{-1,2}(\Omega) \to \mathcal{L}^1_{k_\bfx,0}(\triax)$ defined in~\eqref{eq:def-Pi} pointwise in time. 
More precisely, let $b_{\bfx,i} \in \mathcal{L}^1_{k_\bfx,0}(\triax)$ denote the Bernstein basis function and let $\psi_{\bfx,i} \in \mathcal{L}^1_{3k_\bfx,0}(\triax)$ denote the biorthogonal basis defined in Section~\ref{sec:higher-order} for all interior Lagrange nodes $i\in \nodesint_\bfx$ of $\triax$. 
The interpolation operator in space reads
\begin{align}\label{eq:DefPix}
  \Pi_\bfx v = \sum_{i \in \nodesint_\bfx} \langle v,\psi_{\bfx,i}\rangle_\Omega \,b_{\bfx,i} \qquad \text{  for all }v \in V. 
\end{align}
The composition of both operators
$
\Pi_\otimes \coloneqq \Pi_\bft \Pi_\bfx \colon V\to V_h$ reads
\begin{align}\label{eq:Composition}
  \Pi_\otimes v = \sum_{i \in \nodest}\sum_{j\in \nodesint_\bfx} \int_{\mathcal{J}} \langle v(s) , \psi_{\bft,i}(s) \psi_{\bfx,j}\rangle_{\Omega} \ds\, b_{\bft,i} b_{\bfx,j}\qquad\text{for all }v\in V.
\end{align}

\begin{lemma}[Commutation]\label{lem:Commutation}
  The interpolation operators satisfy for all functions $\xi\in W^{1,2}(\mathcal{J};W^{-1,2}(\Omega))$, $v\in L^2(\mathcal{J};W^{1,2}_0(\Omega))$, and $w \in V$ 
  \begin{align}\label{eq:Commutation}
    \partialt \Pi_\bfx \xi  = \Pi_\bfx \partialt \xi,\qquad
    \nablax \Pi_\bft v = \Pi_\bft \nablax v,\qquad
    \Pi_\bft \Pi_\bfx w  = \Pi_\bfx \Pi_\bft w.
  \end{align}
\end{lemma}
\begin{proof}
  For smooth functions the commutation properties follow directly from the identities in~\eqref{eq:DefPIt},~\eqref{eq:DefPix}, and~\eqref{eq:Composition}. 
  For general functions in the respective function spaces the identities follow by density arguments and the stability properties of $\Pi_t$ and $\Pi_x$ according to Theorem~\ref{thm:interpolOperator}.  
\end{proof}
The properties of $\Pi_\bfx$ and $\Pi_\bft$ lead to the following stability result.
\begin{theorem}[Stability]\label{thm:stabPiotimes}
  We have for all $v\in V$, $K = K_\bft \times K_\bfx \in \mathcal{Q}$, and $m,\ell \in \mathbb{N}_0$
  \begin{align*}
    \norm{\partialt^m \nablax^\ell \Pi_\otimes v}_{L^2(K_\bft;L^2(K_\bfx))} 
    & \lesssim \norm{\partialt^m \nablax^\ell v}_{L^2(\omega_{K_\bft};L^2(\omega_{K_\bfx}))},\\
    \norm{\partialt^m \Pi_\otimes v}_{L^2(K_\bft;W^{-1,2}(\Omega))} 
    & \lesssim \norm{\partialt^m v}_{L^2(\omega_{K_\bft};W^{-1,2}(\Omega))}. 
  \end{align*}
\end{theorem}
\begin{proof}
  Let $v \in V$ and let $m,\ell \in \mathbb{N}_0$. 
  By the stability properties of $\Pi_\bfx$ in Theorem~\ref{thm:interpolOperator} and $\Pi_\bft$ in Theorem~\ref{thm:InterpolInTime} using the commutation properties in Lemma~\ref{lem:Commutation} we have that
  \begin{align*}
  & \norm{\partialt^m \Pi_\otimes v}_{L^2(K_\bft;W^{-1,2}(\Omega))} 
     = \norm{ \partialt^m \Pi_\bft \Pi_\bfx v}_{L^2(K_\bft;W^{-1,2}(\Omega))}\\
    &  \lesssim \norm{ \partialt^m \Pi_\bfx v }_{L^2(\omega_{K_\bft};W^{-1,2}(\Omega))} 
      = \norm{ \Pi_\bfx  \partialt^m v}_{L^2(\omega_{K_\bft};W^{-1,2}(\Omega))} 
      \lesssim \norm{ \partialt^m v }_{L^2(\omega_{K_\bft};W^{-1,2}(\Omega))}.
  \end{align*} 
  A similar calculation leads to local stability in the $L^2$-norms. 
\end{proof}
The approximation properties of $\Ix$ and $\It$ result in the following estimates. 
Let $\mathcal{V}_x$ denote the set of vertices in $\triax$ and let $\omega_{x,j} \coloneqq \bigcup \lbrace T \in \triax\colon j \in T\rbrace$ and 
$\omega_{x,j}^2 \coloneqq \bigcup\lbrace T\in \triax \colon T \cap \omega_{x,j} \neq \emptyset  \rbrace$ denote the spatial nodal patches for $j\in \mathcal{V}_x$. 
\begin{theorem}[Approximation on tensor meshes]\label{thm:ApxTensorMeshes}
  We have for all $v\in V$, all time-space cells $K = \Kt\times \Kx\in \mathcal{Q}$, and all integers $0\leq m\leq  s\leq k_\bft + 1$ the estimate 
  \begin{align} \label{eq:ApxIotimes1}
    \begin{aligned}
      &\norm{ \partialt^m (v - \Pi_\otimes v)}_{L^2(\Kt;W^{-1,2}(\Omega))} 
      \lesssim h_\Kt^{s-m} \norm{\partialt^{s} v}_{L^2(\omega_\Kt;W^{-1,2}(\Omega))} \\
      &\qquad  \qquad + \Big( \sum_{j \in \mathcal{V}_x} \norm{\partialt^m v-\Pi_x \partialt^m v}_{L^2(\omega_\Kt;W^{-1,2}(\omega_{x,j}))}^2\Big)^{1/2}.
    \end{aligned}
  \end{align}
  The local interpolation error in space with vertex $j \in \mathcal{V}$ are bounded for integers $0 \leq r \leq k_\bfx + 1$  by 
  \begin{align}\label{eq:interSpaceTimeEasy}
    \begin{aligned}
      \norm{\partialt^m v-  \Pi_x\partialt^m v}_{L^2(\omega_\Kt;W^{-1,2}(\omega_{x,j}))} & \lesssim  \norm{\partialt^m v}_{L^2(\omega_\Kt;W^{-1,2}(\omega^2_{x,j}))},\\
      \norm{\partialt^m v-\Pi_x \partialt^m v }_{L^2(\omega_\Kt;W^{-1,2}(\omega_{x,j}))} & \lesssim h_\Kx^{r+1} \norm{\partialt^m \nablax^{r} v}_{L^2(\omega_\Kt;L^2(\omega_{x,j}^2))}.
    \end{aligned}
  \end{align}
  Furthermore, we have the local $L^2$-estimate for all integers $0\leq m \leq s \leq k_t+1$ and $0\leq \ell \leq r \leq k_x+1$ as
  \begin{align}\label{eq:ApxIotimes2}
    \begin{aligned}
      &\norm{ \partialt^m \nablax^\ell (v-\Pi_\otimes v)}_{L^2(\Kt;L^2(\Kx))} \\
      &\qquad \lesssim h_\Kt^{s-m} \norm{ \partialt^{s} \nablax^\ell v }_{L^2(\omega_\Kt;L^2(\Kx))} + h_\Kx^{r-\ell} \norm{ \partialt^m \nablax^r v }_{L^2(\omega_\Kt;L^2(\omega_\Kx))}.
    \end{aligned}
  \end{align}
\end{theorem}
\begin{proof}
  Let $v\in V$ and let $\Kt\in \triat$.
  The local estimates in~\eqref{eq:interSpaceTimeEasy} 
  follow directly by Theorem~\ref{thm:interpolOperator}. 
  The triangle inequality yields with $\delta \coloneqq v - \Pi_\bft v$ that
  \begin{align}\label{eq:TriangleProof}
    \begin{aligned}
      &\norm{ \partialt^m (v - \Pi_\otimes v) }_{L^2(\Kt;W^{-1,2}(\Omega))} 
      \leq \norm{ \partialt^m (v - \Pi_\bfx v) }_{L^2(\Kt;W^{-1,2}(\Omega))} \\
      &\qquad  + \norm{ \partialt^m (v - \Pi_\bft v)}_{L^2(\Kt;W^{-1,2}(\Omega))} + \norm{ \partialt^m ( \delta - \Pi_\bfx\delta) }_{L^2(\Kt;W^{-1,2}(\Omega))}.
    \end{aligned}
  \end{align}
  Due to Theorem~\ref{thm:interpolOperator}, the commutation property in Lemma~\ref{lem:Commutation}, and the stability of $\Pi_t$ in Theorem~\ref{thm:InterpolInTime} the third addend  satisfies for $v \in W^{s,2}(\omega_{\Kt};W^{-1,2}(\Omega))$ that
  \begin{align*}
    & \norm{ \partialt^m ( \delta - \Pi_\bfx\delta) }^2_{L^2(\Kt;W^{-1,2}(\Omega))} 
      = \norm{ \partialt^m \delta - \Pi_\bfx\partialt^m \delta }_{L^2(\Kt;W^{-1,2}(\Omega))}^2\\
    & \eqsim \sum_{i\in \mathcal{V}_x} \norm{ \partialt^m \delta - \Pi_\bfx\partialt \delta }_{L^2(\Kt;W^{-1,2}(\omega_{x,j}))}^2 \lesssim \sum_{i\in \mathcal{V}_x} \norm{ \partialt^m v - \Pi_\bfx \partialt^m v }_{L^2(\omega_\Kt;W^{-1,2}(\omega_{x,j}))}^2.
  \end{align*}
The first term on the right-hand side of~\eqref{eq:TriangleProof} can be bounded  similarly.
By Theorem~\ref{thm:InterpolInTime} we obtain for the second term and all $0\leq m\leq  s\leq k_\bft + 1$ that
  \begin{align*}
    \norm{ \partialt^m (v-\Pi_t v)}_{L^2(\Kt;W^{-1,2}(\Omega))} \lesssim h_{K_\bft}^{s-m} \norm{ \partialt^{s} v}_{L^2(\omega_\Kt;W^{-1,2}(\Omega))}.
  \end{align*}
  Combining the previous estimates proves~\eqref{eq:ApxIotimes1}.
  Similar arguments prove the local $L^2$-estimate in~\eqref{eq:ApxIotimes2}.
\end{proof}
\begin{remark}[Applications]
  The discretization $V_h$ in~\eqref{eq:TensorFEspace} is used in \cite{StevensonWesterdiep20} and \cite{DieningStorn21}.
 The extension of the result to adaptively refined meshes is an open problem. 
 Such an adaptive refinement might result in unstructured simplicial meshes as in \cite{LangerSteinbachTroltzschYang21} or in varying spatial meshes associated to certain time steps \cite{Stevenson2022}. 
\end{remark}
\begin{remark}[Comparison]
  The projection operator $\Pi_\textup{FK} \colon V \to \mathcal{L}^1_1(\triat) \otimes \mathcal{L}^1_{1,0}(\triax)$ for tensor-product meshes in \cite[Thm.~7]{FuehrerKarkulik19} combines Lagrange interpolation in time and the $L^2$-projection $\Pi_2$ in space. 
This leads to an operator that is local in time but non-local in space. 
Let $h_\bfx \coloneqq \max_{\Kx \in \triax} h_\Kx$ denote the maximal mesh size in $\triax$, then this operator satisfies for all $v\in V$ and  any $\Kt\in \triat$
  \begin{align*}
    \norm{ \nablax (v - \Pi_\textup{FK} v) }_{L^2(\Kt,L^2(\Omega))} & \lesssim h_{K_\bft} \norm{ \partialt \nablax v }_{L^2(\Kt,L^2(\Omega))} + h_\bfx \norm{ \nablax^2 v}_{L^\infty(\Kt,L^2(\Omega))},\\
    \norm{ \partialt  (v - \Pi_\textup{FK} v) }_{L^2(\Kt;W^{-1,2}(\Omega))} & \lesssim h_{K_\bft} \norm{ \partialt^2 v }_{L^2(\Kt;W^{-1,2}(\Omega))} + \frac{h^2_\bfx}{h_\Kt} \norm{ \nablax v }_{L^\infty(\Kt;L^2(\Omega))}.
  \end{align*}
  This result requires more regular solutions than the estimate in Theorem~\ref{thm:ApxTensorMeshes}. 
\end{remark}

\subsection{Smoothing rough right-hand sides}\label{sec:SmoothingRHS}
Besides in parabolic problems, functions in negative Sobolev spaces occur as right-hand sides of elliptic PDEs. 
Such rough right-hand sides cause severe difficulties for several established numerical schemes like non-conforming, DPG, or DG methods. 
Indeed, they may lead to an ill-posed system or to a reduced rate of convergence compared to conforming schemes. 
Applying a smoothing operator to the right-hand side, as for example in~\cite{VeeserZanotti18,VeeserZanotti18b,VeeserZanotti19}, is one possible remedy. 
We illustrate the use of the operator $\Pi$ defined in Section~\ref{sec:higher-order} as such a smoothing operator by applying it to a least squares finite element method. 
We thereby complement the result in \cite[Sec.~3.3]{FHK.2022} for the lowest-order case $k=0$, by the higher-order cases $k\geq 1$. 
For this we use the discrete space defined below in~\eqref{eq:discSpaceXh}. 
For the simplicity of presentation we consider the Poisson equation as model problem.
Note that the ideas presented extend to a much wider class of PDEs as illustrated in \cite{CarstensenStorn18}. 

Given a function $f\in W^{-1,2}(\Omega)$ we seek the solution $u\in W^{1,2}_0(\Omega)$ to 
\begin{align}\label{eq:Poisson}
  -\Delta u = f \quad \text{ in }W^{-1,2}(\Omega).
\end{align}
We reformulate the problem as first order system given by
\begin{align}\label{eq:PMPfosls}
  \textup{div}\, \sigma + f = 0 \;\;\text{ in }W^{-1,2}(\Omega)\qquad\text{and}\qquad \nabla u - \sigma = 0 \;\; \text{ in }L^2(\Omega).
\end{align}
Let the space $H(\textup{div},\Omega)$ and the Raviart--Thomas space for $k\in \mathbb{N}_0$ be given by
\begin{align*}
  H(\textup{div},\Omega) & \coloneqq \lbrace \tau \in L^2(\Omega)^d \colon \textup{div}\, \tau \in L^2(\Omega)\rbrace,\\ 
  \RTk &\coloneqq \lbrace \tau_h \in H(\textup{div},\Omega) \colon  \tau_h(x)|_T \in \mathcal{P}_k(T)^d + x \mathcal{P}_k(T)\text{ for all }T\in \tria\rbrace.
\end{align*}
The general idea of first-order system least square methods (FOSLS) consists in minimizing the residuals of~\eqref{eq:PMPfosls} in the squared $L^2$-norms over discrete spaces as
\begin{align}\label{eq:discSpaceXh}
  X_h \coloneqq \mathcal{L}^1_{k+1,0}(\tria) \times \RTk.
\end{align}
However, the $L^2$-norm of the first residual is not well-defined for non square-integrable  $f\in W^{-1,2}(\Omega)\setminus L^2(\Omega)$. 
We remedy this issue by using the operator $\Pi\colon W^{-1,2}(\Omega)\to \mathcal{L}^1_{k}(\tria)$ as a smoother and instead compute the minimizer 
\begin{align}\label{eq:LSFEMreg}
  (\tilde{u}_h,\tilde{\sigma}_h) = \argmin_{(v_h,\tau_h)\in X_h} \lVert \textup{div}\, \tau_h + \Pi f\rVert_{L^2(\Omega)}^2 + \lVert \nabla v_h - \tau_h \rVert_{L^2(\Omega)}^2.
\end{align}
\begin{theorem}[A~priori estimate]
  Let $k\in \mathbb{N}$.
  With the minimizers $\tilde{u}_h,\tilde{\sigma}_h$ in~\eqref{eq:LSFEMreg} and the solution $u$ to~\eqref{eq:Poisson} one has for all $(s_j)_{j\in \mathcal{V}} \subset \lbrace 0 ,\dots, k+1 \rbrace$ that
  \begin{align*}
    &\lVert \nabla (u - \tilde{u}_h) \rVert_{L^2(\Omega)}^2 + \lVert \nabla u - \tilde{\sigma}_h \rVert_{L^2(\Omega)}^2\\
    &\qquad\lesssim \sum_{j\in\mathcal{V}} h_j^{2s_j} \lVert f \rVert^2_{W^{-1+s_j,2}(\omega^2_j)} + \min_{(v_h,\tau_h) \in X_h} \big(\lVert \nabla (u - v_h) \rVert_{L^2(\Omega)}^2 + \lVert \nabla u - \tau_h \rVert_{L^2(\Omega)}^2\big). 
  \end{align*}
  The hidden constant depends solely on the shape regularity of $\tria$ and the domain $\Omega$.
\end{theorem}
\begin{proof}
  Let $\tilde{u} \in W^{1,2}_0(\Omega)$ solve the auxiliary problem $-\Delta \tilde{u} = \Pi f$ in $W^{-1,2}(\Omega)$.
  Applying the triangle inequality yields 
  \begin{align*}
    &\lVert \nabla (u-\tilde{u}_h) \rVert_{L^2(\Omega)} + \lVert \nabla u - \tilde{\sigma}_h \rVert_{L^2(\Omega)} \\
    &\qquad \leq 2 \lVert \nabla (u-\tilde{u}) \rVert_{L^2(\Omega)} + \lVert \nabla ( \tilde{u} - \tilde{u}_h)\rVert_{L^2(\Omega)} + \lVert \nabla \tilde{u} - \tilde{\sigma}_h\rVert_{L^2(\Omega)}.
  \end{align*}
  The latter two addends are bounded due to the quasi-optimality \cite[Thm.~5.30]{BochevGunzburger09} of the LSFEM with right-hand side $\Pi f$ by
  \begin{align*}
    &\lVert \nabla ( \tilde{u} - \tilde{u}_h)\rVert_{L^2(\Omega)} + \lVert \nabla \tilde{u} - \tilde{\sigma}_h\rVert_{L^2(\Omega)}+  \lVert \textup{div}\,\nabla\tilde{u} - \textup{div}\,\tilde{\sigma}_h\rVert_{L^2(\Omega)}\\
    &\quad \lesssim \min_{(v_h,\tau_h) \in X_h} \big(\lVert \nabla ( \tilde{u} - v_h)\rVert_{L^2(\Omega)} + \lVert \nabla \tilde{u} - \tau_h \rVert_{L^2(\Omega)} + \lVert \textup{div}\,\nabla\tilde{u} - \textup{div}\,\tau_h\rVert_{L^2(\Omega)}\big).
  \end{align*}
  By the triangle inequality the first term can be bounded as
  \begin{align*}
    \min_{v_h \in \mathcal{L}^1_{k+1,0}(\tria)} \lVert \nabla ( \tilde{u} - v_h)\rVert_{L^2(\Omega)} \leq \lVert \nabla (u-\tilde{u}) \rVert_{L^2(\Omega)} + \min_{v_h \in \mathcal{L}^1_{k+1,0}(\tria)}\lVert \nabla (u-v_h)\rVert_{L^2(\Omega)}.
  \end{align*}
  Let $\Pi_2^\textup{pw}\colon L^2(\Omega) \to \mathcal{L}^0_{k}(\tria)$ be the $L^2$-orthogonal projection onto piece-wise polynomials of degree $k\in \mathbb{N}$. 
  The interpolation operator $\mathcal{I}_\mathcal{RT}\colon H(\textup{div},\Omega) \to \RTk$ in \cite[Thm.~23.12]{ErnGuermond21} satisfies $\Pi_2^\textup{pw} \textup{div}\,\sigma = \textup{div}\,\mathcal{I}_\mathcal{RT}\sigma$ and
  \begin{align*}
    \lVert \sigma - \mathcal{I}_\mathcal{RT}\sigma\rVert_{L^2(\Omega)} \lesssim \min_{\tau_h \in \RTk} \lVert \sigma - \tau_h \rVert_{L^2(\Omega)}\qquad\text{for all }\sigma \in H(\textup{div},\Omega).
  \end{align*}
  Since $-\textup{div}\, \nabla \tilde{u} = \Pi f \in \mathcal{L}^1_{k}(\tria)\subset \mathcal{L}^0_{k}(\tria)$, the application of $\mathcal{I}_\mathcal{RT}$ to $\sigma \coloneqq \nabla \tilde{u}$ shows
  \begin{align*}
    &\min_{\tau_h \in \RTk} \lVert \nabla \tilde{u} - \tau_h \rVert_{L^2(\Omega)} + \lVert \Pi f + \textup{div}\,\tau_h\rVert_{L^2(\Omega)}\\
    &\qquad \leq \lVert \nabla \tilde{u} - \mathcal{I}_\mathcal{RT}\nabla \tilde{u} \rVert_{L^2(\Omega)} + \lVert -\textup{div}\, \nabla \tilde{u} + \textup{div}\,\mathcal{I}_\mathcal{RT}\nabla \tilde{u}\rVert_{L^2(\Omega)}\\
  &\qquad \lesssim \min_{\tau_h \in \RTk} \lVert \nabla \tilde{u} - \tau_h \rVert_{L^2(\Omega)}. 
  \end{align*}
  Theorem~\ref{thm:interpolOperator} bounds  
  $\lVert \nabla (u-\tilde{u}) \rVert_{L^2(\Omega)} =\lVert f - \Pi f \rVert_{W^{-1,2}(\Omega)}$. Combining all estimates concludes the proof.
\end{proof}

\appendix
\section{Averaged Taylor polynomial in Bochner spaces}
\label{apx:verificationLemPolyApxInTime}
In this appendix we introduce the averaged Taylor polynomial in Bochner spaces $L^p(I;X)$ for an open interval $I$ of size~$h_t>0$ and a Banach space~$X$.
We adapt the techniques employed in~\cite[Ch.~4]{BS.2008} to our need.
Recall that $\mathcal{P}_s(I) \otimes X$ is the space of $X$-valued polynomials of order at most~$s \in \setN_0$. 
Let $\eta \in C^\infty_0(I)$ be a normalized density function with $\norm{ \eta}_{L^1(I)} = 1$ and $\norm{\partialt^m \eta }_{L^\infty(I)} \lesssim h_t^{-m-1}$ for any $m\in \mathbb{N}_0$. 
Such a function can be obtained by scaling and translation of the standard bump function on the unit interval.
For $v \in W^{s,1}(I;X)$ with $s\in \mathbb{N}$ we define the averaged Taylor polynomial $T^{s} v \in \mathcal{P}_s(I)\otimes X$ for all $\tau \in I$ by 
\begin{align}\label{eq:AvgTaylor}
  \begin{aligned}
    (T^s v) (\tau)& \coloneqq \sum_{\ell =0}^s \frac{1}{\ell!} \int_{I} (\partialt^\ell v)(\sigma)\;(\tau-\sigma)^\ell \eta (\sigma)\dsig\\
    &\hphantom{:} = \sum_{\ell =0}^s \frac{(-1)^\ell}{\ell!} \int_{I} v(\sigma)\, \partialt^\ell \big((\tau-\sigma)^\ell\eta (\sigma)\big)\dsig.
  \end{aligned}
\end{align}
The second formula allows us to define the polynomial for functions~$v \in L^1(I;X)$.
\begin{lemma}[Averaged Taylor polynomial]
  \label{lem:PolyApxInTime}
  The averaged Taylor polynomial of order $s\in \mathbb{N}_0$ has the following properties.
  \begin{enumerate}
  \item \label{itm:abgTay2} For all $v_s \in \mathcal{P}_s(I)\otimes X$ we have the identity $v_s = T^s v_s$.
  \item \label{itm:abgTay3} It holds $\partialt^\ell T^s w = T^{s-\ell} \partialt^\ell w$ for all $\ell \in \mathbb{N}_0$ with $\ell \leq s$ and $w\in W^{\ell,1}(I;X)$.
  \item \label{itm:avgTayStab} For $0\leq m \leq s$ and all $v \in L^1(I;X)$ there holds
    \begin{align*}
      \norm{\partial_t^m T^s v}_{L^\infty(I;X)} &\lesssim  h_t^{-1}                     \norm{\partial_t^m v}_{L^1(I;X)}.
    \end{align*}
  \item \label{itm:avgTayStabp} For $p\in [1,\infty]$, $0\leq m \leq s$, and all $v \in W^{m,p}(I;X)$ there holds
    \begin{align*}
      \norm{\partial_t^m T^s v}_{L^p(I;X)} &\lesssim                      \norm{\partial_t^m v}_{L^p(I;X)}.
    \end{align*}
  \item \label{itm:avgTayApprox} For $p\in [1,\infty]$, $0 \leq m \leq n \leq s+1$, and all $v\in W^{n,p}(I;X)$ there holds
    \begin{align*}
      \norm{ \partialt^m (v - T^s v) }_{L^p(I;X)} \lesssim h_t^{n-m} \norm{ \partialt^n v }_{L^p(I;X)}.
    \end{align*}
  \end{enumerate}
 The hidden constants solely depend on $s$.
\end{lemma}
\begin{proof}
  Let $s\in \mathbb{N}_0$, $p\in [1,\infty]$, and $v\in L^1(I;X)$. 
  Then~\ref{itm:abgTay2} and~\ref{itm:abgTay3} follow exactly as in Proposition~4.1.9 and 4.1.17 in~\cite{BS.2008}, respectively.
  We estimate for $\tau \in I$
  \begin{alignat*}{2}
    \norm{(T^s v) (\tau)}_X
    &\leq\sum_{\ell =0}^s \frac{1}{\ell!} \int_{I} \norm{v(\sigma)}_X \, \bigabs{\partialt^\ell \big((\tau-\sigma)^\ell\eta (\sigma)\big)}\dsig
      \\
    &\leq \norm{v}_{L^1(I;X)} \sum_{\ell =0}^s \frac{1}{\ell!} \sum_{k=0}^s \binom{\ell}{k} h_t^{k} \norm{\partial_t^k \eta}_{L^\infty(I)}
    &&\leq h_t^{-1} \norm{v}_{L^1(I;X)}.
  \end{alignat*}
  This shows $\norm{T^s v}_{L^\infty(I;X)} \lesssim h_t^{-1} \norm{v}_{L^1(I;X)}$ and thus verifies~\ref{itm:avgTayStab} for~$m=0$. 
  An application of~\ref{itm:abgTay3} proves the general case $m=0,\ldots, s$. 
  H\"older's inequality yields
  \begin{align*}
    \norm{\partial_t^m T^s v}_{L^p(I;X)}
    &\lesssim h_t^{\frac{1}{p}}
      \norm{\partial_t^m T^s v}_{L^\infty(I;X)}
      \lesssim h_t^{\frac 1p -1}
      \norm{\partial_t^m v}_{L^1(I;X)}
      \leq
   \norm{\partial_t^m v}_{L^p(I;X)}.
  \end{align*}
  This proves~\ref{itm:avgTayStabp}.
  Let $z \in X^*$, then for all~$\tau \in I$
  \begin{align*}
    \skp{(T^sv)(\tau)}{z}_{X \times X^*} = T^s(\skp{v}{z}_{X \times X^*})(\tau).
  \end{align*}
  We assume that $v \in W^{s+1}(I;X)$ and conclude for the remainder $R^s v \coloneqq v - T^s v$ by Proposition~4.2.8 of~\cite{BS.2008} that for almost all~$\tau \in I$ we have
  \begin{align*}
    \bigabs{\skp{(R^sv)(\tau)}{z}_{X \times X^*}}
    &= \bigabs{R^s(\skp{v(\tau)}{z}_{X \times X^*})}
    \\
    &\lesssim \int_I \abs{\tau-\sigma}^s \abs{\partial_t^{s+1} \skp{v(\sigma)}{z}_{X \times X^*}} \dsig
    \\
    &\lesssim \int_I \abs{\tau-\sigma}^s \norm{\partial_t^{s+1} v(\sigma)}_X \dsig\, \norm{z}_{X^*}.
  \end{align*}
  Taking the supremum over all~$z \in X^*$ with  $\norm{z}_{X^*} \leq 1$ implies
  \begin{align*}
    \norm{(R^sv)(\tau)}_X
    &\lesssim \int_I \abs{\tau-\sigma}^s \norm{\partial_t^{s+1} v(\sigma)}_X \dsig
    \lesssim h_t^s \norm{\partial_t^{s+1} v}_{L^1(I;X)}.
  \end{align*}
  This proves that
  \begin{align*}
    \norm{v - T^s v}_{L^\infty(I;X)} =
    \norm{R^s v}_{L^\infty(I;X)}
    &\lesssim h_t^s \norm{\partial_t^{s+1} v}_{L^1(I;X)}.
  \end{align*}
  Thus, applying H\"older's inequality we obtain 
  \begin{align*}
    \norm{v - T^s v}_{L^p(I;X)}
    &\leq h_t^{\frac 1p} \norm{R^sv}_{L^\infty(I;X)}\\
     & \leq h_t^{s+\frac 1p} \norm{\partial_t^{s+1} v}_{L^1(I;X)}
      \lesssim h_t^{s+1} \norm{\partial_t^{s+1} v}_{L^p(I;X)}.
  \end{align*}
  This proves~\ref{itm:avgTayApprox} for $m=0$ and $n = s+1$. 
For general $m=0,\dots, s$ we calculate
  \begin{align*}
    \norm{\partial_t^m (v- T^s v)}_{L^p(I;X)}
    &= \norm{(\partial_t^m v)- T^{s-m} (\partial_t^m v)}_{L^p(I;X)}
    \lesssim h_t^{s+1-m} \norm{\partial_t^{s+1} v}_{L^p(I;X)}.
  \end{align*}
  This proves~\ref{itm:avgTayApprox} for $n=s+1$. For $n \leq s$ we conclude with $T^s T^{n-1} v = T^{n-1} v$ that
  \begin{align*}
    \norm{\partial_t^m (v- T^s v)}_{L^p(I;X)}
    &\leq
    \norm{\partial_t^m (v- T^{n-1} v)}_{L^p(I;X)} +
    \norm{\partial_t^m T^s(v- T^{n-1} v)}_{L^p(I;X)}.
  \end{align*}
  The stability~\ref{itm:avgTayStabp} and the already proven case of~\ref{itm:avgTayApprox} show that
  \begin{align*}
    \norm{\partial_t^m (v- T^s v)}_{L^p(I;X)}
    &\lesssim
      \norm{\partial_t^m (v- T^{n-1} v)}_{L^p(I;X)} \lesssim
      h_t^{n-m}
    \norm{\partial_t^n v}_{L^p(I;X)}.
  \end{align*}
  This proves the claim.
\end{proof}

\printbibliography


\end{document}